\documentclass{amsart}
\usepackage[utf8]{inputenc}
\usepackage{amsmath}
\usepackage{amssymb}
\usepackage{amsthm}
\usepackage{xcolor}
\usepackage{StandardNewCommands}
\usepackage{tikz}
\usepackage{mathtools}
\usetikzlibrary{cd}

\usepackage[backref=true,giveninits=true,doi=false,url=false,isbn=false,backend=biber]{biblatex}
    \newbibmacro{string+doiurlisbn}[1]{%
      \iffieldundef{doi}{%
        \iffieldundef{url}{%
          \iffieldundef{isbn}{%
            \iffieldundef{issn}{%
              #1%
            }{%
              \href{http://books.google.com/books?vid=ISSN\thefield{issn}}{#1}%
            }%
          }{%
            \href{http://books.google.com/books?vid=ISBN\thefield{isbn}}{#1}%
          }%
        }{%
          \href{\thefield{url}}{#1}%
        }%
      }{%
        \href{http://dx.doi.org/\thefield{doi}}{#1}%
      }%
    }
    \DeclareFieldFormat{title}{\usebibmacro{string+doiurlisbn}{\mkbibemph{#1}}}
    \DeclareFieldFormat[article,incollection,thesis,misc,inproceedings,online,inbook]{title}{\usebibmacro{string+doiurlisbn}{\mkbibquote{#1}}}
    \DeclareFieldFormat{year}{\usebibmacro{year-parenthesis}{\mkbibemph{#1}}}
\renewbibmacro*{date}{%
\iffieldundef{year}
{}
{\ifentrytype{misc}{\printtext[parens]{\printdate}}{\printdate}}}%
\usepackage{hyperref}
\hypersetup{colorlinks=true,allcolors=blue}
\usepackage[capitalize,nameinlink,noabbrev]{cleveref}

\newtheorem{theorem}{Theorem}[section]
\newtheorem{corollary}[theorem]{Corollary}
\newtheorem{lemma}[theorem]{Lemma}
\newtheorem{proposition}[theorem]{Proposition}

\theoremstyle{definition}
\newtheorem{definition}[theorem]{Definition}
\newtheorem*{acknowledgements}{Acknowledgements}
\newtheorem{remark}[theorem]{Remark}

\newtheorem{example}[theorem]{Example}

\DeclareMathOperator{\Ext}{{\rm Ext}}
\DeclareMathOperator{\Hom}{{\rm Hom}}
\DeclareMathOperator{\Spec}{{\rm Spec}}

\newcommand{\ZnZconst}[1]{\underline{\ZZ/{#1}\ZZ}}

\title[Semistable abelian varieties over $\QQ$ with bad reduction at 19 only]{Semistable abelian varieties over $\QQ$ with bad reduction at 19 only: an overview of the Fontaine--Schoof strategy}

\date{}

\author{Francesco Campagna}

\address{Université Clermont Auvergne - LMBP, UMR 6620 - CNRS, Campus des Cézeaux 3, place Vasarely 63178 Aubière cedex, France}
\email{francesco.campagna@uca.fr}
\urladdr{https://sites.google.com/view/francesco-campagna/}

\author{Pip Goodman}

\address{Departament de matem\`atiques i inform\`atica and Centre de recerca matem\`atica,
Universitat de Barcelona,
Gran via de les Corts Catalanes 585, 08007 Barcelona, Catalonia}
\email{pip.goodman@ub.edu}
\urladdr{https://pipgoodman.github.io/}

\begin{document}

\begin{abstract}
    In this paper we provide an overview of a strategy pioneered by Fontaine and heavily refined by Schoof to classify abelian varieties with prescribed bad reduction.
    Throughout the overview, we prove various non-trivial background results turning it into an introduction for readers unacquainted with this topic.

    With the overview completed, we provide explicit examples of the strategy in action.
    At first we give introductory examples, classifying semistable abelian varieties over $\QQ$ with bad reduction at exactly one of 3 or 5 up to isogeny over $\QQ$.
    We then move onto a harder example, proving the analogous result for 19, which is new.
\end{abstract}

\maketitle

\tableofcontents
\section{Introduction}
In 1985, Fontaine \cite{Fontaine_pas_de_variete_abelienne} proved that there are no abelian varieties with good reduction everywhere over $\QQ$.
Building upon his work, Brumer and Kramer \cite{Brumer_Kramer_non-existence_of_certain_semistable_ab_vars} proved the non-existence of semistable abelian varieties with good reduction outside of $N \in \{2,3,5,7\}$.
 Schoof \cite{Schoof_one_bad_prime} reproved Brumer and Kramer's result and extended it to the case $N=13$.
Moreover, via a substantial extension to Fontaine's method, he showed \cite{Schoof_one_bad_prime,Schoof_23, Schoof_15}  that for $N \in \{11,15,23\}$, any such abelian variety is isogenous to a power of $J_0(N)$, the Jacobian of the modular curve $X_0(N)$. 

These classification results were initially used by Khare and Wintenberger to prove Serre's modularity conjecture for low levels and weights \cite[Thm. 5.2, Thm. 5.6]{Khare_Wintenberger_initial_breakthrough}.
These results later served as crucial base cases for their full resolution of Serre's modularity conjecture \cite{Khare_Wintenberger_I,Khare_Wintenberger_II}.

Inspired by the techniques developed by Schoof, Brumer and Kramer were able to rule out the existence of certain semistable abelian surfaces with odd conductor.
This, coupled with the computation of paramodular cusp forms \cite{Poor_Yuen_paramodular_cusp_forms} by Poor and Yuen, lead them to their paramodularity conjecture \cite{Paramodularity_Conjecture}.

In this paper we give an introduction to Fontaine and Schoof's classification strategy.
Our goal here is to make this overview accessible to readers who have had a basic introduction to finite flat group schemes, as given, for example, in Tate's article \cite{Tate_FLT}.

As a means of illustrating their strategy, we reprove that there are no semistable abelian varieties with good reduction outside exactly one of 3 or 5, implying, in particular, Fontaine's theorem that there are no abelian varieties with  good reduction everywhere over $\QQ$.
In order to demonstrate more advanced theory, we prove:

\begin{theorem}[= Theorem \ref{thm:19_classification}] \label{thm:main_theorem}
    Any semistable abelian variety over $\mathbb{Q}$ that has good reduction outside of 19 is isogenous to a power of $J_0(19)$.
\end{theorem}

We begin by recalling a few basic definitions and results about finite flat group schemes in Section \ref{section:preliminaries}.
With this completed we progress to our overview of Fontaine and Schoof's strategy in Section \ref{Section:FS_strategy}.
Here we give detailed arguments for results, presumably well-known to experts, which are unfortunately not easily accessible in the literature.
This does however force us a little outside of Tate's article in Section \ref{subsection:extensions}, where we need to assume a little background about sheaves on Grothendieck topologies, which one can find, for example, in Tamme's book \cite{Tamme_etale_cohomology}.
Readers learning this material, or about finite flat group schemes, for the first time may also wish to consult \cite{Waterhouse_group_schemes_book,Shatz_grp_schms,Tate_p-divisible_groups,Raynaud_quotient,Demazure_Gabriel,SGA3_tome_I}.

With the overview finished, in Section \ref{sec:FS_examples} we prove that there are no semistable abelian varieties with good reduction outside exactly one of $3$ or $5$.
This section is signposted with steps occurring in the Fontaine--Schoof strategy, allowing it to be read alongside Section \ref{Section:FS_strategy}.

The reader who only wishes to read the material necessary for these introductory examples may take a slightly faster route through Section \ref{Section:FS_strategy}.
Namely,
Subsection \ref{subsec:reduction_to_Tate_modules},
Subsection \ref{sec:strategy_generic_fibre},
Proposition \ref{prop:grp_schms_order_2} from Subsection \ref{sec:strategy_prolongations},
Subsection \ref{subsection:extensions} up to Corollary \ref{Cor:if_n_kills_the_base_change_then_it_kills_you_too} and up to Proposition \ref{prop:filtration_has_no_Z/lZ_nor_mu_l} in Subsection \ref{subsec:classification_Tate_modules}.

Section \ref{sec:19} is dedicated to proving Theorem \ref{thm:main_theorem} and requires the full force of the techniques developed in Section \ref{Section:FS_strategy}.
In Section \ref{sec:simple_objects_of_order_4}, we expand upon a remark of Schoof \cite[pg. 861]{Schoof_one_bad_prime} analysing certain simple objects of order 4 in $\mathcal{C}_{p}$.
This section may be viewed as a natural follow-up to Subsubsection \ref{sec:group_schemes_prime_order}.
The proof of Proposition \ref{prop:PGR_2-torsion_can_give_objects_in_C_p} and its corollaries have few prerequisites, namely no more than what is covered in Subsection \ref{subsec:reduction_to_Tate_modules}.
Likewise, Proposition \ref{prop:family_of_simple_grp_schemes_order_4_existence} and its proof may be read after the reader has encountered Proposition \ref{prop:equivalence_categories_prolongations}. 

 The interested reader looking to learn more about abelian varieties with prescribed bad reduction over $\QQ$ may wish to consult the papers of Mestre \cite{Mestre_formules_explicites} and Calegari \cite{Calegari_semistable_ab_vars}.
Wherein conditional non-existence results for certain abelian varieties defined over $\QQ$ with good reduction at primes greater than 10 are obtained.

Finally, we mention that in a forthcoming article, the authors of this paper prove the analogue of Theorem \ref{thm:main_theorem} for $N=29$.

\begin{acknowledgements}
The authors express their gratitude to René Schoof for answering many questions related to this topic and for giving inspiring talks on the subject.
The authors extend their thanks to Remy van Dobben de Bruin and Jakob Stix for helpful conversations and the anonymous referees for their valuable feedback and suggestions.

This work started whilst the authors were postdoctoral researchers at the Max-Planck-Institut für Mathematik in Bonn.
The authors thank all the institutions which hosted them whilst they carried out work on this project: Max-Planck-Institut für Mathematik, Universität Hannover, Universität Bayreuth, Université Clermont Auvergne, Universitat de Barcelona.

The authors collectively received funding from CNRS under the PEPS JCJC program 2024 and 2025, the ANR-20-CE40-0003 Jinvariant, the Deutsche Forschungsgemeinschaft (DFG) project grant STO 299/18-2 (AOBJ: 686837) and by the Spanish Ministry of Science and Innovation via the grant “Abelian varieties, L-functions, and rational points” (code PID2022-137605NB-I00).
\end{acknowledgements}
\section{Preliminary results on finite flat group schemes}
\label{section:preliminaries}

We begin by recalling a few basic definitions, as can be found in \cite{Tate_FLT}.
Let $S$ be a scheme and $(\mathrm{Sch}/S)$ be the category of schemes over $S$. We will call any object of this category an \textit{$S$-scheme}, and in the case $S=\mathrm{Spec}(R)$ is affine, we will also talk of $R$-schemes (or schemes over $R$). 

\begin{definition}
    A group scheme $G$ over $S$ is a commutative group object in $(\mathrm{Sch}/S)$.
    That is, $G$ is an object in $(\mathrm{Sch}/S)$ with a morphism $G \times G \rightarrow G$ such that for all objects $T$ in $(\mathrm{Sch}/S)$ the induced morphism $G(T) \times G(T) \rightarrow G(T)$ gives $G(T)$ the structure of a commutative group.
\end{definition}

In particular, all group schemes are assumed to be commutative.

\begin{definition}
    A homomorphism of group schemes over $S$ is a homomorphism of group objects in $(\mathrm{Sch}/S)$, i.e., a morphism $\varphi \colon G \rightarrow G'$ of $S$-schemes such that for any $T \in (\mathrm{Sch}/S)$, the induced map $G(T) \rightarrow G'(T)$ given by $g \mapsto \varphi \circ g$ is a homomorphism of groups.
\end{definition}

\begin{definition}
    A finite flat group scheme $G$ over $S$ is a group scheme over $S$ such that $G \rightarrow S$ is finite and flat.
\end{definition}

From now on we assume $S$ is a locally noetherian scheme and $G$ is a finite flat group scheme over $S$.
This implies $G \rightarrow S$ is finite locally free \cite[Prop. 12.19]{Goertz_Wedhorn_book}.
In particular, we can speak of the order of $G$, that is, the degree of $G \rightarrow S$ as defined in \cite[pg. 333]{Goertz_Wedhorn_book}.
The order of $G$ is a locally constant function $S \rightarrow \ZZ_{\geq 0}$.
Hence in the case $S$ is connected, the order of $G$ is a constant.
If $S$ is connected, then given a prime number $\ell$, we say that $G$ is an \textit{$\ell$-group scheme} if its order is a power of $\ell$.

\begin{definition}\label{def:exact_sequence_group_schemes}
    Let $F,G$ and $H$ be finite flat group schemes over $S$.
    The sequence of group scheme homomorphisms 
    \[
    0\to F \xrightarrow[]{\iota} G \xrightarrow[]{p} H \to 0
    \]
    is said to be \textit{exact} if the following holds true:
    \begin{enumerate}
        \item $F = \ker(p)$, where the kernel is meant in the sense of category theory;
        \item the morphism $p$ is faithfully flat.
    \end{enumerate}
\end{definition}

\begin{remark}
\label{remark:sheaf_injectivity_same_as_closed_embedding}
    The statement (1) in the definition above is equivalent to requiring that 
\[
0\to F(A) \xrightarrow[]{\iota_A} G(A) \xrightarrow[]{p_A} H(A)
\]
is left exact for every $S$-scheme $A$.

We also note that $F \rightarrow G$ is a closed immersion.
Indeed, since $H \rightarrow S$ is finite, it is separated \cite[pg. 586]{Goertz_Wedhorn_book}, thus the identity section $S \rightarrow H$ is a closed immersion \cite[\href{https://stacks.math.columbia.edu/tag/045W}{Tag 045W}]{stacks-project}.
As base change preserves the property of being a closed immersion \cite[pg. 583]{Goertz_Wedhorn_book}, we find $\ker(p) \rightarrow G$ is a closed immersion.
\end{remark}

Recall that finite morphisms are affine by definition \cite[Definition 12.9]{Goertz_Wedhorn_book}.
In particular, any finite flat group scheme $G$ over an affine base is affine.

\begin{lemma} \label{lem:correspondence_submodules_subschmes}
Let $R$ be a Dedekind domain with field of fractions $K$ of characteristic 0, and let $G$ be a finite flat group scheme over $R$.
Then the map that sends each closed finite flat subgroup scheme $H \subseteq G$ over $R$ to the $\mathrm{Gal}(\overline{K}/K)$-module $H(\overline{K})$, induces a one-to-one correspondence
\[\{\text{Closed finite flat subgroup schemes of $G$}\} \leftrightarrow \{ \mathrm{Gal}(\overline{K}/K) -\text{submodules of } G(\overline{K}) \}.\]
Moreover the order of $H$ equals the cardinality of $H(\overline{K})$.
\end{lemma}

\begin{proof}
Write $G= \Spec(A)$, where the Hopf algebra $A$ is regarded as an $R$-module via the map induced by the structural morphism $G\to \mathrm{Spec}(R)$.
Then the base-change of $G$ to $K$ is $G_K \coloneqq \Spec(A_K)$ where $A_K=A \otimes_R K$ inherits a natural Hopf algebra structure from that of $A$. Since $A$ is flat over $R$, we have a natural injection $A\hookrightarrow A_K$ and we identify $A$ with its image under this map.

We begin by showing that the set $\mathcal{J}$ of ideals $J$ of $A$ such that the quotient $A/J$ is flat over $R$ is in bijection with the set $\mathcal{I}$ of ideals of $A_K$. Consider the maps $\varphi\colon \mathcal{J} \to \mathcal{I}$ and $\psi\colon \mathcal{I} \to \mathcal{J}$ defined by
\[
\varphi\colon J\mapsto J \otimes_R K, \hspace{0.5cm} \psi: I\mapsto I \cap  A.
\]

The map $\psi$ is well-defined, \textit{i.e.}, for every ideal $I\subseteq A_K$ the quotient ring $A/(I \cap A)$ is flat over $R$.
Indeed, by \cite[Cor. 1.2.14, pg. 11]{Liu_book} and the fact that $R$ is a Dedekind domain, this is equivalent to $0$ being the only torsion element of $A/(I \cap A)$.
Let $a \in A$ be a lift of a torsion element.
Then there exists $r \in R\setminus\{0\}$ such that $ra \in I \cap A \subseteq I$.
Moreover, as $I$ is an $A \otimes_R K$ ideal, we have $a \in I$ and thus $a \in I \cap A$.
Hence $A/(I \cap A)$ is torsion-free and $\psi$ is well-defined.

We now claim that the maps $\varphi$ and $\psi$ are mutually inverses of one another.
Let $J$ be an ideal of $A$ such that $A/J$ is flat.
It is clear that $(J \otimes_R K) \cap A  \supseteq J$. Conversely, let $x=\sum_{i=1}^n a_i \otimes k_i \in (J \otimes_R K) \cap A$, where $a_i\in J$ and $k_i\in K$ for all $i=1,...,n$. Since $R$ is a Dedekind domain, there exists $r\in R\setminus \{0\}$ such that $rk_i\in R$ for all $i=1,...,n$. Then we have
\[
r\cdot x= \sum_{i=1}^n a_i \otimes r k_i = \sum_{i=1}^n (r k_i) a_i \otimes 1 \in J
\]
since $J$ is an ideal of $A$. As above, flatness of the module $A/J$ implies it is torsion-free, and we deduce that $x \in J$. This shows that $(J \otimes_R K) \cap A=J$.

Now let $I$ be an ideal of $A_K$. It is clear that $(I \cap A) \otimes_R K \subseteq I$. Conversely, let $x=\sum_{i=1}^n a_i \otimes k_i \in I \subseteq A_K$, where $a_i\in A$ and $k_i\in K$ for all $i=1,...,n$. As before, choose $r\in R\setminus \{0\}$ such that $rk_i\in R$ for all $i=1,...,n$. Then we have
\[
r\cdot x=\sum_{i=1}^n a_i \otimes r k_i= \sum_{i=1}^n (r k_i) a_i \otimes 1 \in I \cap A
\]
since $I$ is an ideal. Hence we obtain
\[
x=\sum_{i=1}^n (r k_i) a_i \otimes r^{-1} \in (I \cap A) \otimes_R K
\]
and this shows that $(I \cap A) \otimes_R K = I$. In particular, the maps $\varphi$ and $\psi$ are mutual inverses, and this proves our claim. 

It is now not very difficult to see that $\varphi$ and $\psi$ induce a bijection between Hopf ideals of $A$ with flat quotient and Hopf ideals of $A_K$ (see \cite[pg. 14]{Waterhouse_group_schemes_book} for the definition of a Hopf ideal).
To give an idea of how this can be proved, let us show that if $I$ is a Hopf ideal of $A_K$ then $J\coloneqq \psi(I) = I \cap A$ is in the kernel of the map
\[
(\Delta \ \mathrm{mod} \ J) \colon  A\xrightarrow[]{\Delta} A\otimes_R A \xrightarrow[]{} A/J \otimes_R A/J
\]
where $\Delta$ is the co-multiplication of $A$ and the second arrow is induced by the natural projection modulo $J$.
We shall leave the rest of the verifications to the reader.

Since $\varphi$ is inverse to $\psi$, we have $I=J\otimes_R K$ and by hypothesis $I$ is in the kernel of 
\[
(\Delta \ \mathrm{mod} \ J) \otimes \mathrm{id}_K\colon  A_K \to (A_K/I) \otimes_K (A_K/I) \cong (A/J \otimes_R A/J) \otimes_R K.
\]
In particular, for all $a\in J$ we have $(\Delta(a) \text{ mod } J)\otimes 1=0$.
Since $A/J \otimes _R A/J$ is a flat $R$-module, we deduce that $(\Delta(a) \text{ mod } J)=0$ for all $a\in J$ and this is what we wanted to prove.

The bijective correspondence given by $\varphi$ and $\psi$ described above shows, after taking ring spectra, that there is a bijection between finite flat closed subgroup schemes of $G$ and finite (flat) closed subgroup schemes of $G_K$.
Since $K$ has characteristic 0, the group scheme $G_K$ is étale over $K$ \cite[pg. 138 (II)]{Tate_FLT} so in turn its closed subgroup schemes correspond bijectively to the $\mathrm{Gal}(\overline{K}/K)$-submodules of $G_K(\overline{K})=G(\overline{K})$ by \cite[pg. 137]{Tate_FLT}. This completes the proof of the first statement.

The second statement about orders follows from the following more general discussion: if $H=\mathrm{Spec}(B)$ is a finite flat group scheme over $R$ (which being an integral domain is thus connected), then $B$ is a locally free $R$-algebra of constant rank $n$, hence for every prime ideal $\mathfrak{p}$ of $ R$ we have an $R_\mathfrak{p}$-module isomorphism $B_\mathfrak{p} \cong R_\mathfrak{p}^n$.
Taking $\mathfrak{p}=(0)$ we obtain a $K$-vector space isomorphism $B_{(0)} \cong B\otimes_R K \cong K^n$, showing $H$ and $H_K$ have the same order.
Now, as before $H_K$ is étale over $K$ since the latter has characteristic 0.
Its order then equals $H(\overline{K})$ by \cite[beginning of pg. 137]{Tate_FLT}.
This concludes the proof of the lemma.
\end{proof}

The above result does not hold in positive characteristic.
For example $\mu_p/\FF_p$ is a non-trivial group scheme of order $p$, but $\mu_p(L)$ has a unique point for any field extension $L$ of $\FF_p$.

\begin{example}
\label{example:closedbutnotflatsubgroupscheme}
Let us give an example of a closed subgroup scheme which is not flat.
Let $R = \ZZ[\frac{1}{N}]$ for $N$ an odd integer.
The group schemes $ \mu_2$ and $ \underline{\ZZ/2\ZZ} $ considered over $R$ are determined by the Hopf algebras $B \coloneqq R[x]/\langle x^2-1 \rangle $ and $A \coloneqq Re_0 \oplus Re_1$  respectively, where $e_0$ and $e_1$ are orthogonal idempotents.
Let $\varphi$ denote the non-trivial homomorphism $\underline{\ZZ/2\ZZ} \rightarrow \mu_2$ induced by the Hopf algebra homomorphism $B \rightarrow A$ given by extending $x \mapsto e_0-e_1$ to a homomorphism of $R$-algebras.
The kernel of $\varphi$ is represented by $A / 2e_1A$ which is not flat since it has non-trivial $R$-torsion.
On the other hand, the kernel of $\varphi$ is a closed subgroup scheme of $\underline{\ZZ/2\ZZ}$.
\end{example}

\begin{remark}
\label{remark:not_abelian}
    The above example shows that the full subcategory of the category of group schemes consisting of finite flat group schemes of 2 power order is not abelian.
\end{remark}

\section{The Fontaine--Schoof classification strategy}
\label{Section:FS_strategy}
The aim of this section is to describe the strategy of the proof of Theorem \ref{thm:main_theorem}. The proof is divided into several steps, each addressed in a separate subsection. At each step, besides indicating the pattern of the proof, we state and prove some basic results that are relevant for the computations involved in the step itself. The results that we decided to prove in the various subsections are certainly well-known to experts, but we either could not find an adequate reference for them, or we decided that they were of significant importance to be included for the convenience of the reader. 

Some conventions and notation that will be used throughout the paper: we fix once and for all an algebraic closure $\overline{\mathbb{Q}}$ of the rational field $\mathbb{Q}$ and we set $\Gamma_\mathbb{Q}\coloneqq\mathrm{Gal}(\overline{\mathbb{Q}}/\mathbb{Q})$.
For a finite flat group scheme $G$, defined over a subring of $\QQ$, we write $\QQ(G)$ to denote the field cut out by the action of $\Gamma_\QQ$ on $G(\overline{\QQ})$.
If $K,L$ are subfields of $\overline{\mathbb{Q}}$ we denote by $KL$ their compositum.
For $S$-schemes $X\to S$ and $T\to S$, we denote by $X_T\coloneqq X\times_S T$ the base change of $X$ along $T\to S$.
When $S$ and $T$ are affine, we use the corresponding rings (whose spectra are $S$ and $T$) in the notation.

\subsection{Reduction to the classification of Tate modules}
\label{subsec:reduction_to_Tate_modules}

The following result is fundamental to the Fontaine--Schoof strategy.
\begin{proposition}
\label{prop:torsion_is_finite_flat}
    Let $A/\QQ$ be an abelian variety with good reduction outside of an integer $N$ and let $\mathcal{A}$ be its Néron model over $\mathbb{Z}$.
    
    Then for any integer $n$, coprime to $N$, the $n$-torsion group scheme $\mathcal{A}[n]$ restricts to a finite flat group scheme over $\ZZ[\frac{1}{N}]$ which becomes étale over $\ZZ[\frac{1}{nN}]$.
\end{proposition}

\begin{proof}
The Néron model $\mathcal{A} \to \mathrm{Spec} (\mathbb{Z})$ restricts to an abelian scheme over $\mathbb{Z} [\frac{1}{N}]$.
By \cite[Section 7.3, Lemma 2 (a)]{BLR_book_1990}, the multiplication-by-$n$ map $\mathcal{A} \xrightarrow[]{\cdot n} \mathcal{A}$ is a flat and quasi-finite morphism over $\mathbb{Z} [\frac{1}{N}]$.
Base changing the multiplication-by-$n$ map by the identity section $\Spec(\ZZ[\frac{1}{N}]) \rightarrow \mathcal{A}$, we find the group scheme $\mathcal{A}[n]$ is flat and quasi-finite over $\mathbb{Z} [\frac{1}{N}]$, since these properties are preserved under base change \cite[pg. 583, 584]{Goertz_Wedhorn_book}.
On the other hand, $\mathcal{A}_{\mathbb{Z} [\frac{1}{N} ]}$ being an abelian scheme is proper, and separated because it is the Néron model over $\mathbb{Z} [\frac{1}{N} ]$ of its generic fibre.
So by \cite[Prop. 12.58 (3)]{Goertz_Wedhorn_book} the multiplication-by-$n$ map is a proper endomorphism of $\mathcal{A}_{\mathbb{Z} [\frac{1}{N} ]}$.
This, together with the quasi-finiteness property, implies that $\mathcal{A}[n]$ is in fact finite over $ \mathbb{Z} [\frac{1}{N} ]$ by \cite[Cor. 12.89]{Goertz_Wedhorn_book}.

Finally, $\mathcal{A}[n]_{\mathbb{Z} [\frac{1}{nN} ]}$ is étale by \cite[Section 7.3, Lemma 2 (b)]{BLR_book_1990}.
\end{proof}

The situation where we find ourselves in this paper is the following: we have an explicit semistable abelian variety $A/\QQ$ with bad reduction at a unique prime $p$, and we would like to prove that any other semistable abelian variety $A'/\mathbb{Q}$ with bad semistable reduction just at $p$ is isogenous to a power of $A$.
The way we proceed, following the strategy pioneered by Fontaine and further developed by Schoof, is to study the $\ell$-adic Tate module $T_\ell(A')$ for a fixed prime $\ell \neq p$.
If we manage to show that there exists $g \in \mathbb{Z}_{\geq 1}$ and an isomorphism of $\Gamma_\mathbb{Q}$-modules
\begin{equation}\label{eq:Tate_modules_isomorphism}
    T_\ell(A')\otimes_{\mathbb{Z}} \mathbb{Q}_\ell \cong T_\ell(A^g)\otimes_{\mathbb{Z}} \mathbb{Q}_\ell
\end{equation}
then by Faltings' Isogeny Theorem \cite[Kor. 2 (ii)]{Faltings_1983} we have that $A'$ is isogenous to $A^g$.

To prove isomorphism \eqref{eq:Tate_modules_isomorphism}, one can of course show the stronger $\Gamma_\mathbb{Q}$-module isomorphism $T_\ell(A') \cong T_\ell(A^g)$ and this in turn is equivalent to having compatible $\Gamma_\mathbb{Q}$-module isomorphisms $A'[\ell^n](\overline{\mathbb{Q}}) \cong A^g[\ell^n](\overline{\mathbb{Q}})$ for all $n \in \mathbb{Z}_{>0}$.
The idea now is to observe that, as shown in Proposition \ref{prop:torsion_is_finite_flat}, we are studying the Galois action on the points of the generic fibre of some group schemes over $\mathbb{Z}$ that become finite and flat over $\mathbb{Z}[\frac{1}{p}]$. Moreover, thanks to work of Grothendieck, the action of inertia at $p$ is constrained, see Lemma \ref{lem:filtration_l_torsion}.
This imposes numerous restrictions on the possible $\Gamma_\mathbb{Q}$-module structures at play. One may even wonder whether, after a good choice for the prime $\ell$, for any semistable abelian variety $A'/\QQ$ with good reduction outside of $p$, the $\ell$-divisible group associated to its Néron model is isomorphic to the $\ell$-divisible group associated to some power of the Néron model of $A$.    

To capture the aforementioned constraints, it will be convenient to introduce the following category.

\begin{definition}
\label{def:categories_of_semistable_objects}
    Let $\ell,N \in \mathbb{\ZZ}_{>0}$ with $\ell$ prime, not dividing $N$. 
Let $\mathcal{C}_{N, \ell}$ be the full subcategory of the category of group schemes over the ring $\ZZ[\frac{1}{N}]$, with objects consisting of finite flat  group schemes $G$ of $\ell$-power order for which the inertia groups $I_p\subseteq \mathrm{Gal}(\mathbb{Q}(G)/\mathbb{Q})$ relative to some prime lying above those dividing $N$ satisfy $(\sigma-1)^2=0$ on $G(\overline{\mathbb{Q}})$ for all $\sigma \in I_p$.

For $p$ a prime different from $\ell$, we write $\mathcal{C}_p \coloneqq \mathcal{C}_{p,2}$.
\end{definition}

We shall use the following lemma repeatedly in the rest of the paper.
Recall that a finite flat group scheme $H$ over a base $S$ is  \textit{simple} if it does not contain proper non-trivial closed finite  flat subgroup schemes over $S$.

\begin{lemma}
\label{lemma:properties_annihiliated_by_l_without_Fontaine}
    Let $H \in \mathcal{C}_{N, \ell}$ be an object annihilated by $\ell$.
    Let $\Gamma = \Gal(\QQ(H)/\QQ)$.
    The extension $\QQ(H)/\QQ$ enjoys the following properties:
        \begin{enumerate}
        \item[(a)] it is unramified outside $ N \ell$;
        \item[(b)] all inertia groups relative to primes lying above primes dividing $N$ have order dividing $\ell$.
    \end{enumerate}
    In particular, the field of points of any simple object in $\mathcal{C}_{N, \ell}$ satisfies the above properties.
\end{lemma}

\begin{proof}
    The Galois extension $\QQ(H)/\QQ$ is finite and unramified outside $\ell N$ since $H$ is étale over $\ZZ[\frac{1}{\ell N}]$ by \cite[pg. 138, (II)]{Tate_FLT}.     

As $H \in \mathcal{C}_{N, \ell}$, for any $\sigma \in I_p$ we have $(\sigma-1)^2=0$ on $H(\overline{\QQ})$, thus in particular, $(\sigma-1)^\ell =0$.
Hence, by the binomial expansion $(\sigma-1)^\ell = \sigma^\ell -1 =0$ as $H$ is annihilated by $\ell$.
It follows that the extension $\QQ(H)/\QQ$ is tamely ramified at $p$, so $I_p$ is cyclic and hence of order dividing $\ell$.

It rests to show that a simple finite flat group scheme $H$ of $\ell$-power order is necessarily annihilated by $\ell$.
This follows from noting that $\ell \cdot H$ is a proper subgroup scheme of $H$ by Lemma \ref{lem:correspondence_submodules_subschmes}, since $\ell \cdot H(\overline{\QQ})$ is a proper submodule of $H(\overline{\QQ})$.
\end{proof}

The next lemma outlines some key properties of the group schemes $\mathcal{A}[\ell^n]$ for $n \in \mathbb{Z}_{>0}$, showing in particular, that they belong to $\mathcal{C}_{N, \ell}$.

\begin{lemma} \label{lem:filtration_l_torsion}
Let $A$ be a semistable abelian variety over $\mathbb{Q}$ with good reduction outside of $N$.
Let $\mathcal{A}\to \Spec(\ZZ[\frac{1}{N}])$ be the Néron model of $A$.
Then for every non-zero $\ell$ coprime to $N$ and $n \in \mathbb{Z}_{>0}$ the $\ell^n$-torsion group scheme $\mathcal{A}[\ell^n]$ admits a filtration 
\[
\mathcal{A}[\ell^n] = G_0 \supseteq G_1 \supseteq \ldots \supseteq G_m = 0
\]
where each $G_i$ is a closed finite flat subgroup scheme of $\mathcal{A}[\ell^n]$ over $\ZZ[\frac{1}{N}]$ and the quotients $H_i\coloneqq G_i/G_{i+1}$ are simple. Moreover, the following holds true:
\begin{enumerate}
    \item For all $i\in \{0,1,...,m\}$ and primes $p $ dividing $ N$, the inertia groups $I_p\subseteq \mathrm{Gal}(\mathbb{Q}(G_i)/\mathbb{Q})$ relative to primes lying above $p$ satisfy for all $\sigma \in I_p$ that $(\sigma-1)^2=0$ on $G_i(\overline{\mathbb{Q}})$.
    In particular, they are cyclic of $\ell$-power order.
    \item For all $i\in \{0,1,...,m\}$ and primes $p$ dividing $ N$, the inertia groups $I_p \subseteq \mathrm{Gal}(\mathbb{Q}(H_i)/\mathbb{Q})$ relative to primes lying above $p$ are cyclic of order dividing $\ell$ and for all $\sigma \in I_p$ we have $(\sigma-1)^2=0$ on $H_i(\overline{\mathbb{Q}})$.
\end{enumerate}
In particular, each $G_i$ and $H_i$ is an object of $\mathcal{C}_{N,\ell}$
\end{lemma}
 
\begin{proof}
The group scheme $\mathcal{A}[\ell^n]$ is finite and flat over $\ZZ[\frac{1}{N}]$ by Proposition \ref{prop:torsion_is_finite_flat}.
As $\mathcal{A}[\ell^n](\overline{\QQ})$ is a finite $\Gamma_\QQ$-module, it follows from Lemma \ref{lem:correspondence_submodules_subschmes} that $\mathcal{A}[\ell^n]$ admits a filtration
\[
\mathcal{A}[\ell^n] = G_0 \supseteq G_1 \supseteq \ldots \supseteq G_m = 0
\]
by closed finite flat subgroup schemes $G_i$ over $\ZZ[\frac{1}{N}]$ with successive simple quotients $H_i$ being finite flat  $\ell$-group schemes.

Let $p$ be a prime dividing $N$.
By \cite[Exp. IX, Prop. 3.5]{Grothendieck_book_1972}, we have that for every prime $\mathfrak{p} $ of $\QQ(\mathcal{A}[\ell^n])$ lying above $p$ the corresponding inertia group $I(\mathfrak{p}/p) \subseteq \mathrm{Gal}(\mathbb{Q}(\mathcal{A}[\ell^n])/\mathbb{Q})$ acts on the $\overline{\mathbb{Q}}$-points of $\mathcal{A}[\ell^n]$ by elements $\sigma$ satisfying $(\sigma-1)^2=0$.
In particular it follows that $I(\mathfrak{p}/p)$ is cyclic of $\ell$-power order \cite[Exp. IX, Cor. 3.5.2]{Grothendieck_book_1972}.
By definition of group action on submodules and quotients, the same properties hold for the inertia groups $I_p$ at $p$ in each $\Gal(\QQ(G_i)/\QQ)$ and $\Gal(\QQ(H_i)/\QQ)$.
Thus $G_i$ and $H_i$ are objects of $\mathcal{C}_{N,\ell}$.
The remaining claim now follows from Lemma \ref{lemma:properties_annihiliated_by_l_without_Fontaine}.
\end{proof}

Lemma \ref{lem:filtration_l_torsion} shows us that any $\ell$-divisible group which arises from a semistable abelian variety with good reduction outside of $N$, is made up of objects in $\mathcal{C}_{N, \ell}$.
Thus to classify such $\ell$-divisible groups we may start by studying those  coming from objects in $\mathcal{C}_{N, \ell}$.
This in turn starts with classifying the simple objects of $\mathcal{C}_{N, \ell}$.

The classification of simple objects $H$ in $\mathcal{C}_{N, \ell}$  follows a two-step strategy.
First of all, we classify the possible generic fibres $H_\mathbb{Q}$.
Since finite flat group schemes over fields of characteristic 0 are étale \cite[pg. 138, (II)]{Tate_FLT}, this is equivalent to classifying the corresponding $\Gamma_\mathbb{Q}$-modules $H(\overline{\mathbb{Q}})$ \cite[pg. 137]{Tate_FLT}. 
These modules, being the representations attached to a simple object in $\mathcal{C}_{N, \ell}$, must respect various ramification constraints that will be further discussed in Section \ref{sec:strategy_generic_fibre}.

Once we have the classification of the generic fibres $H_\mathbb{Q}$, one has to understand in how many ways $H_\mathbb{Q}$ can be prolonged to $\mathbb{Z}[\frac{1}{N}]$ in order to get an object of $\mathcal{C}_{N, \ell}$.
The main tools used for this are Proposition \ref{prop:grp_schms_order_2} and results from Raynaud's paper \cite{Raynaud_type_p}, as discussed in Section \ref{sec:strategy_prolongations}.

Having determined the simple objects of $\mathcal{C}_{N, \ell}$,
we then examine what $\ell$-divisible groups they can give rise to.
In doing this one first makes a few preliminary restrictions on the $\ell$-group schemes that can occur in an $\ell$-divisible group coming from an abelian variety.
In this article, sufficient restrictions on such $\ell$-group schemes will follow from classifying extensions of simple objects in $\mathcal{C}_{N, \ell}$ by one another.
With these preliminary restrictions in hand, one then proceeds to show that all possible $\ell$-divisible groups must come from a known abelian variety.

To sum up, the strategy that we will follow in order to classify semistable abelian varieties over $\mathbb{Q}$ with good reduction outside of $N$ can be summarised in the following steps: 

\begin{enumerate}
    \item Classification of the generic fibre of simple objects in $\mathcal{C}_{N, \ell}$ for some $\ell \nmid N$;
    \item Prolongation of the generic fibres obtained in Step 1 to determine all simple objects in $\mathcal{C}_{N, \ell}$;
    \item Classify extensions in $\mathcal{C}_{N, \ell}$ of its simple objects by one another.
    \item Classification of the $\ell$-divisible groups over $\mathbb{Z}[\frac{1}{N}]$ attached to semistable abelian varieties over $\mathbb{Q}$ with good reduction outside of $N$.
\end{enumerate}

 Note that in order to perform the steps above one does not a priori need any hypothesis on the prime $\ell \nmid N$.
 In practice, however, to carry out the first step, one is often forced to choose $\ell=2$.
 The reason for this, as we will see, comes in when trying to bound the degree of the extensions generated by the generic fibres of simple objects in $\mathcal{C}_{N, \ell}$.

It is notable, however, that Fontaine does not use $\ell =2$.
His reason for this is that for an odd prime $\ell$, the category $\mathcal{C}_{N, \ell}$ is abelian \cite[Thm. 2, pg. 527]{Fontaine_pas_de_variete_abelienne}.
Whereas, as noted in Remark \ref{remark:not_abelian}, this is not true of $\mathcal{C}_{N,2}$.
Since we are forced to work in $\mathcal{C}_{N, 2}$, this a priori throws out the tools of homological algebra, which would be near essential in Step 3 of the above.
Thankfully, it is possible to overcome these technical difficulties as explained in Section \ref{subsection:extensions}.

\subsection{Step 1: The generic fibre of simple objects in $\mathcal{C}_{p, \ell}$} \label{sec:strategy_generic_fibre}

For the sake of simplicity, we shall restrict to the case that $N=p$ is a prime (different from $\ell$) from now on.
Let $H$ be a simple object of $\mathcal{C}_{p, \ell}$. As explained above, classifying the possibilities for the generic fibre $H_\mathbb{Q}$ up to isomorphism is equivalent to classifying the possible $\Gamma_\mathbb{Q}$-modules $H(\overline{\mathbb{Q}})$ up to isomorphism.
In fact, since $H$ is a simple finite flat group scheme, this reduces to classifying the irreducible such modules as $H(\overline{\mathbb{Q}})$ is simple by Lemma \ref{lem:correspondence_submodules_subschmes}.

In Lemma \ref{lemma:properties_annihiliated_by_l_without_Fontaine}, we saw $H(\overline{\QQ})$ is unramified outside $p\ell$  and that the inertia groups corresponding to primes above $p$ were heavily restricted.
A theorem of Fontaine (which we state as Theorem \ref{thm:Fontaine} below) provides restrictions on the higher ramification groups at $\ell$.
This result is more naturally expressed in the local setting, to which we can always reduce after completing $\QQ(H)$ at some prime above $\ell$.

We state Fontaine's theorem, starting with the necessary setup and a review of fundamental facts about finite extensions of $p$-adic fields.
Let $K$ be a finite extension of $\mathbb{Q}_\ell$ for some prime $\ell$ and let $L$ be a finite Galois extension of $K$ with Galois group $\Gamma$ over $K$.
If $\mathcal{O}_L$ denotes the ring of integers in $L$ and $v_L$ the normalised valuation associated to $L$, then for every $i \in \mathbb{Z}$ with $i\geq -1$ we define 
\[
\Gamma_i\coloneqq  \{\sigma \in G: v_L(\sigma(x)-x) \geq i + 1  \text{ for all } x\in \mathcal{O}_L\} \subseteq \Gamma.
\]
The group $\Gamma_i$ is the \textit{$i$-th (higher) ramification group in lower numbering}.
One has that $\Gamma_i$ is a normal subgroup of $\Gamma = \Gamma_{-1}$.
The groups $\Gamma_0=I(L/K)$ and $\Gamma_1$ are respectively the inertia   and the wild inertia group of the extension $L/K$.
Moreover, by \cite[Prop. IV.6, Prop. IV.7]{Serre_book_1979}, the quotient $\Gamma_0/\Gamma_1$ is cyclic of order coprime to $\ell$ while $\Gamma_1$ is an $\ell$-group. 

One can rescale the numbering of the ramification groups as follows: first of all, we extend the lower numbering from $\mathbb{Z}_{\geq -1}$ to the real interval $[-1, + \infty)$ by setting $\Gamma_u \coloneqq \Gamma_{\lceil u \rceil}$, where $\lceil u \rceil$ denotes the smallest integer $\geq u$.
Then, we use the bijective function $\varphi \colon [-1, + \infty) \to [-1, + \infty)$ defined on \cite[pg. 73]{Serre_book_1979} to modify the numbering of the ramification groups. More precisely, we define
\[
\Gamma^u\coloneqq \Gamma_{\varphi^{-1}(u)} \hspace{0.5cm} (\text{or equivalently } \Gamma^{\varphi(u)}=\Gamma_u )
\]
and we call $\Gamma^u$ the \textit{$u$-th ramification group in upper numbering}.
From the definitions it follows that $\Gamma^{-1}=\Gamma=\Gamma_{-1}$ and $\Gamma^0=I(L/K)=\Gamma_0$.
We are now ready to state Fontaine's theorem \cite[\S 2.1, Thm. 1]{Fontaine_pas_de_variete_abelienne}.

\begin{theorem} \label{thm:Fontaine}
Let $\ell$ be a prime, $K/\mathbb{Q}_\ell$ a finite extension of ramification index $e \in \mathbb{Z}_{>0}$ and $\mathcal{O}_K$ the ring of integers of $K$.
Suppose that $H$ is a finite flat group scheme over $\mathcal{O}_K$ that is annihilated by $\ell^n$  for some $n \in \mathbb{Z}_{>0}$.
Then, setting $\Gamma\coloneqq\mathrm{Gal}(K(H)/K)$, we have that $\Gamma^{u}=\{1\}$ for every 
\[
u> e\cdot \left( n + \frac{1}{\ell-1} \right)-1.
\]
\end{theorem}

\begin{corollary}
\label{Cor:properties_of_annihilated_by_ell_with_Fontaine}
    Let $H \in \mathcal{C}_{p, \ell}$ be an object annihilated by $\ell$.
    Let $\Gamma = \Gal(\QQ(H)/\QQ)$.
    The extension $\QQ(H)/\QQ$ enjoys the following properties:
        \begin{enumerate}
        \item[(a)] it is unramified outside $p\ell$;
        \item[(b)] all inertia groups relative to primes lying above $p$ have order dividing $\ell$;
        \item[(c)] the higher ramification groups $\Gamma_\ell^u$ in upper numbering at all primes above $\ell$ are trivial for $u>1/(\ell-1)$.
    \end{enumerate}
        In particular, the field of points of any simple object in $\mathcal{C}_{p, \ell}$ satisfies the above properties.
\end{corollary}

\begin{proof}
This is a direct consequence of Lemma \ref{lemma:properties_annihiliated_by_l_without_Fontaine} and Theorem \ref{thm:Fontaine}.
\end{proof}

Hence, the ramification properties of the extension $ \QQ(H)/\QQ$ are fairly well-understood.
We now aim to examine the minimal (possibly infinite) extension of $\mathbb{Q}$ that contains all finite extensions exhibiting these ramification properties.
The following result establishes the existence of this extension and provides useful information that, at least in the case considered here, aids in its explicit computation.

\begin{lemma} \label{lem:stability_properties_field_extension}
    Let $p, \ell \in \mathbb{Z}$ be distinct primes and $F/\QQ$ a finite Galois extension with Galois group $\Gamma$ satisfying the following properties:
    \begin{enumerate}
        \item[(a)] the extension is unramified outside $p \ell$;
        \item[(b)] all inertia groups relative to primes lying above $p$ have order dividing $\ell$;
        \item[(c)] the higher ramification groups $\Gamma_\ell^u$ in upper numbering at all primes above $\ell$ are trivial for $u>1/(\ell-1)$.
    \end{enumerate}
    Then the following hold true:
    \begin{enumerate}
        \item If $ F'/\QQ$ is another finite Galois extension satisfying (a)--(c), then the extension $ FF'/\QQ$ satisfies (a)--(c).
        \item If $H$ is the narrow Hilbert class field of $F$, then the extension $ H/\QQ$ satisfies (a)--(c). 
        \item If $\Delta_{F/\mathbb{Q}}$ denotes the discriminant of $F/\QQ$, we have
        \[
        v_p(\Delta_{F/\mathbb{Q}}) \leq \frac{\ell-1}{\ell}[F\colon \mathbb{Q}], \hspace{0.5cm} v_\ell(\Delta_{F/\mathbb{Q}}) < \frac{\ell}{\ell-1}[F \colon \mathbb{Q}]
        \]
        where $v_p$ and $v_\ell$ are the normalised valuations on $\mathbb{Q}$ corresponding to the primes $p$ and $\ell$ respectively.
    \end{enumerate}
    
\end{lemma}

\begin{proof}
    Let $ F'/\QQ$ be an extension satisfying properties $(a)$--$(c)$.
    Then the compositum $FF'$ clearly satisfies $(a)$ and it satisfies $(b)$ by Abhyankar's lemma \cite[Thm. 1]{Cornell_1979}. 
    It also satisfies $(c)$ since for all $u\geq-1$ the $u$-th higher ramification groups in upper numbering of $F/\QQ$ and of $ F'/\QQ$ are equal to the image under the natural restriction map of the $u$-th higher ramification group in upper numbering of $ FF'/\QQ$, see \cite[Prop. IV.14]{Serre_book_1979}. This proves the first part of the lemma.

    Let $H$ be the narrow Hilbert class field of $F$.
    Since $F$ is a Galois extension of $\mathbb{Q}$, the extension $H/\QQ$ is also Galois and, by definition of the narrow Hilbert class field, $H/F$ is unramified at all finite primes.
    From this, it follows at once that the extension $H/\QQ$ satisfies $(a)$ and $(b)$.
    Again by \cite[Prop. IV.14]{Serre_book_1979} we have that for all $u>1/(\ell-1)$ the $u$-th higher ramification groups $\mathrm{Gal}(H/\mathbb{Q})_\ell^u$ relative to the primes above $\ell$ are contained in $\mathrm{Gal}(H/F)$.
    Since all inertia subgroups in $\mathrm{Gal}(H/F)$ are trivial, we deduce that $\mathrm{Gal}(H/\mathbb{Q})_\ell^u=\{1\}$ for all $u>1/(\ell-1)$. This shows that $H/\QQ$ also satisfies property $(c)$ and the second part of the lemma is proved.

    Let us now analyse the prime factorisation of the discriminant $\Delta_{F/\mathbb{Q}}$.
    Letting $\mathcal{D}_F$ be the different of $ F/\QQ$, we have that $|\Delta_{F/\mathbb{Q}}|= N_{F/\mathbb{Q}}(\mathcal{D}_F)$ where $N_{F/\mathbb{Q}}$ denotes the ideal norm from $F$ to $\mathbb{Q}$.
    
    By condition $(b)$, the prime $p$ is tamely ramified in the extension $F/\QQ$ and the ramification index $e(\mathfrak{p}/p)$ of each prime $\mathfrak{p} $ of $ F$ lying above $p$ divides $\ell$.
    This ramification index does not depend on the choice of the prime $\mathfrak{p}$ since $F$ is a Galois extension of $\mathbb{Q}$. If we have $e(\mathfrak{p}/p)=1$ for all $\mathfrak{p}$, then $v_p(\Delta_{F/\mathbb{Q}})=0$.
    If this is not the case, then $e(\mathfrak{p}/p)=\ell$ for all $\mathfrak{p}$ dividing $p$ and by \cite[Prop. IV.4]{Serre_book_1979} we have that
    \[
    v_\mathfrak{p}(\mathcal{D}_F) = \ell -1  \hspace{0.5cm} \text{for all } \mathfrak{p}\mid p
    \]
    where $v_\mathfrak{p}$ denotes the normalised $\mathfrak{p}$-adic valuation on $F$.
    Taking norms we obtain
    \[
    v_p(\Delta_{F/\mathbb{Q}})=(\ell-1) \sum_{\mathfrak{p}\mid p} f(\mathfrak{p}/p)=(\ell-1) \frac{[F:\mathbb{Q}]}{\ell}
    \]
    where $f(\mathfrak{p}/p)$ is the residue degree of the prime $\mathfrak{p}$. This proves the desired bound on the $p$-adic valuation of $\Delta_{F/\mathbb{Q}}$.

    We now turn to the study of the $\ell$-adic valuation of this discriminant.
    For this, we note that condition $(c)$ on the ramification filtration in upper numbering relative to primes above $\ell$ can be translated into a bound on the valuation at primes above $\ell$ of $\mathcal{D}_F$ by means of \cite[\S 1.3, Prop.]{Fontaine_pas_de_variete_abelienne}.
    This proposition implies that, in our setting, for all primes $\mathfrak{L}$ of $F$ lying above $\ell$ we have
    \[
    v_\mathfrak{L}(\mathcal{D}_F) < e(\mathfrak{L}/\ell) \cdot \left(\frac{1}{\ell-1} +1\right) = e(\mathfrak{L}/\ell) \cdot  \frac{\ell}{\ell-1}
    \]
    where $e(\mathfrak{L}/\ell)$ denotes the ramification index of $\mathfrak{L}$ above $\ell$.
    Taking norms, we obtain
    \[
    v_\ell(\Delta_{F/\mathbb{Q}}) =\sum_{\mathfrak{L}\mid \ell} f(\mathfrak{L}/\ell) v_\mathfrak{L}(\mathcal{D}_F) < \frac{\ell}{\ell-1} \sum_{\mathfrak{L}\mid \ell} f(\mathfrak{L}/\ell) e(\mathfrak{L}/\ell)= \frac{\ell}{\ell-1} [F:\mathbb{Q}]
    \]
    where $f(\mathfrak{L}/\ell)$ is the residue degree of the prime $\mathfrak{L}$.
    This proves the desired bound on the $\ell$-adic valuation of $\Delta_{F/\mathbb{Q}}$ and concludes the proof.
\end{proof}

\begin{corollary} \label{cor:bound_root_discriminant}
    Let $ F/\QQ$ be a finite Galois extension as in Lemma \ref{lem:stability_properties_field_extension} and denote by $\delta_{F/\mathbb{Q}}$ its root discriminant. Then we have
    \[
    \delta_{F/\mathbb{Q}} < p^{\frac{\ell-1}{\ell}} \cdot \ell^{\frac{\ell}{\ell-1}}.
    \]
\end{corollary}

The smallest extension $L/\QQ$ containing all the finite Galois extensions enjoying properties $(a)$--$(c)$ (and in particular the field $\QQ(H)$ for any simple object $H$ of $\mathcal{C}_{p,\ell}$) is by definition their compositum, which is then Galois over $\mathbb{Q}$.
A priori, the field $L$ may have infinite degree over $\mathbb{Q}$.
However, in certain cases one can use Corollary \ref{cor:bound_root_discriminant} to show that $[L\colon\mathbb{Q}]$ is finite, and even provide an explicit upper bound on its degree.
Indeed, if the upper bound on the root discriminant provided by Corollary \ref{cor:bound_root_discriminant} is sufficiently small (meaning smaller than 21.78), then comparing it with the discriminant lower bounds first obtained by Odlyzko \cite{Odlyzko_I}, \cite{Odlyzko_II} and then improved upon by Poitou \cite{Poitou_1976} provides an upper bound on the degree of each field $F$ satisfying properties $(a)$--$(c)$, and consequently, an upper bound on $[L\colon\mathbb{Q}]$.

In Subsection \ref{subsec:19_step_1} we choose $p = 19$.
Then the only possible choice for the prime $\ell$ that allows us to apply Poitou's bounds, explicitly appearing in Diaz y Diaz's tables \cite{DiazyDiaz_minorations}, is $\ell=2$.
For $p=19$ the upper bound in Corollary \ref{cor:bound_root_discriminant} becomes $\delta_{F/\mathbb{Q}} < 4\sqrt{19} < 17.44$ which implies, using \cite{DiazyDiaz_minorations}, that  $[L\colon\mathbb{Q}] \leq 137$.

A note on the use of Diaz y Diaz's tables \cite{DiazyDiaz_minorations}: these tables are based on the lower bound \cite[Inequality (26)]{Poitou_1976} which is worst possible when, in the setting of that article, the considered number field is totally imaginary.
Thus our upper bounds on $[L\colon\mathbb{Q}]$ are obtained from looking at \cite[Table 1, cas totalement imaginaire, pp. 1-10]{DiazyDiaz_minorations}.  

From the bound on the degree of the extension $L$, we will use techniques from class field theory and representation theory to pin down $L$ sufficiently well to obtain an explicit number field containing $\QQ(H)$ for any simple object $H$ of $\mathcal{C}_{p,\ell}$.
This leads to a finite list of possibilities for the generic fibre $H_\mathbb{Q}$.

\subsection{Step 2: Prolongations} \label{sec:strategy_prolongations}
As before, in this section $\ell$ denotes a prime, $N$ an integer coprime to $\ell$ and $p$ a prime different to $\ell$.

In the first step of the classification strategy, described in Section \ref{sec:strategy_generic_fibre}, we ideally managed to determine all the possibilities for the generic fibre of a simple object in $\mathcal{C}_{N, \ell}$.
We now need to check whether each of these potential generic fibres are actually the generic fibre of such an object, and if so, of how many up to isomorphism.
To do this, we study the \textit{prolongations} of these group schemes which are defined over $\mathbb{Q}$ to finite flat group schemes over $\mathbb{Z} [\frac{1}{N} ]$.

Indeed, a priori, a prolongation of a given group scheme may not exist and, even if it does, it may not be unique.
For instance, the group scheme defined over $\QQ$ whose associated Galois module is trivial of order 2 admits $\mu_2$ and $\underline{\ZZ/2\ZZ}$ as prolongations to $\mathbb{Z} [\frac{1}{N} ]$.
These prolongations are different since $\mu_2$ is connected over $\mathbb{Z} [\frac{1}{N} ]$ while $\underline{\ZZ/2\ZZ}$ is not.  

On the other hand, a group scheme $G/\QQ$ whose associated module $G(\overline{\QQ})$ is unramified outside of an integer $N$ is a module for the étale fundamental group of $\ZZ[\frac{1}{N}]$.
Hence,  $G$ admits a unique prolongation to an étale group scheme $\textbf{G}/\mathbb{Z}[\frac{1}{N}]$ by the equivalence of categories explained in \cite[pg. 137]{Tate_FLT}.
In the above example $\underline{\ZZ/2\ZZ}$ is the unique étale prolongation of $G$ to $\mathbb{Z} [\frac{1}{N} ]$.

This shows that the generic fibre of any object of $\mathcal{C}_{N,\ell}$ has a unique prolongation to $\ZZ[\frac{1}{\ell N}]$, since any such prolongation is étale \cite[pg. 138 (II)]{Tate_FLT}.
To determine the possible prolongations to $\ZZ[\frac{1}{N}]$ we use the following result due to M. Artin.

\begin{proposition} \label{prop:equivalence_categories_prolongations}
    Let $\mathcal{D}_{N,\ell}$ be the category defined by the following data:
    \begin{itemize}
        \item \textbf{Objects:} triples $(G_1,G_2,\theta)$ where $G_1$ is a finite flat group scheme over $\ZZ_\ell$, $G_2$ is a finite flat group scheme over $\mathbb{Z}[\frac{1}{\ell N}]$ and $\theta: (G_1)_{\mathbb{Q}_\ell} \to (G_2)_{\mathbb{Q}_\ell}$ is an isomorphism of group schemes.
        \item \textbf{Morphisms:} A morphism between two triples $(G_1,G_2,\theta)$ and $(G_1',G_2',\theta')$ is the datum of a pair of morphisms of group schemes $\varphi_i\colon G_i \to G_i'$ for $i\in \{1,2\}$ that make the following diagram
        \begin{center}
        \begin{tikzcd}
(G_1)_{\mathbb{Q}_\ell} \arrow[r, "\theta"] \arrow[d, "(\varphi_1)_{\mathbb{Q}_\ell}"'] & (G_2)_{\mathbb{Q}_\ell} \arrow[d, "(\varphi_2)_{\mathbb{Q}_\ell}"] \\
(G_1')_{\mathbb{Q}_\ell} \arrow[r, "\theta'"]                                         & (G_2')_{\mathbb{Q}_\ell}                                        
\end{tikzcd}
\end{center}
commute.
    \end{itemize} 
Then the functor $\mathcal{C}_{N,\ell} \to \mathcal{D}_{N,\ell}$ sending $G$ to the triple $(G_{\mathbb{Z}_\ell}, G_{\mathbb{Z}[\frac{1}{\ell N}]}, \mathrm{id}_{\mathbb{Q}_\ell})$ is an equivalence of categories.
\end{proposition}

\begin{proof}
    This is a special case of \cite[Prop. 2.3]{Schoof_cyclotomic}.
\end{proof}

Given an object $G$ of $\mathcal{C}_p$ such that $G_\QQ$  admits a unique prolongation $G_1/\mathbb{Z}_2$ of $G_{\mathbb{Q}_2}$, and $G_2/\ZZ[\frac{1}{2p}]$ of $G_\QQ$, this proposition implies the  prolongations of $G_\QQ$ to $\mathbb{Z}[\frac{1}{p}]$ are obtained by gluing along isomorphisms $G_1 \times_{\ZZ_2} \Spec(\mathbb{Q}_2) \cong G_2 \times_{\ZZ[\frac{1}{p}]} \Spec(\mathbb{Q}_2)$.

Prolongations of prime order group schemes can be completely classified thanks to the work of Oort and Tate \cite{Oort_Tate}.
Certifying the existence of prolongations for group schemes of non-prime order usually consists in using `explicit' examples.
For instance, the 2-torsion group scheme of a semistable abelian variety with good reduction outside of $p$ may admit subgroup schemes or quotients that prolong one of the group schemes over $\mathbb{Q}$ found in Section \ref{sec:strategy_generic_fibre}.
The game then becomes to show, if possible, that these are the only possible prolongations.

We shall address the question of unicity of prolongations in a more general setting.
Let $R$ be a Dedekind domain with field of fractions $K$ of characteristic 0.
If $G$ is a finite flat group scheme over $K$ and $G_1, G_2$ are two prolongations of $G$ to $R$, we write $G_1 \geq G_2$ if there exists a morphism of group schemes $G_1 \to G_2$ prolonging the identity of $G$.
In the same way as in \cite[Section 2.2]{Raynaud_type_p}, the relation $\geq$ defines a partial order on the set of prolongations of $G$.

\begin{proposition} \label{prop:automorphisms_maximal_prolongation}
    Let $R$ be a Dedekind domain with field of fractions $K$ of characteristic zero.
    Let $G$ be a finite flat group scheme over $K$ and suppose $G$ has a prolongation to $R$.
    Then:
    \begin{enumerate}
        \item The group scheme $G$ admits prolongations $G^+$ and $G^-$ to $R$ that are respectively a maximum and a minimum for the partial order $\geq$;
        \item Any automorphism of $G$ extends uniquely to $G^+$ and $G^-$.
    \end{enumerate}
\end{proposition}
\begin{remark}
    Raynaud \cite{Raynaud_type_p} proves this result only under the assumption that $R$ is a discrete valuation ring.
    The generalisation stated above is straightforward, but we provide a proof for lack of a suitable reference.
\end{remark}
\begin{proof}
Write $G=\mathrm{Spec}(A)$ where $A$ is an Hopf algebra over $K$.
Since $A$ is a separable commutative finite dimensional algebra over $K$, the integral closure $A_R$ of $R$ in $A$ is an $R$-order by \cite[Thm. 10.3]{Reiner_book_1975}.
By assumption on the existence of a prolongation of $G$ to $R$, there exists at least one Hopf $R$-order in $A$, and every such order $\mathcal{O}$ is contained in $A_R$.
We denote by $[A_R:\mathcal{O}]$ the index module of $\mathcal{O}$ in $A_R$.
It is an integral ideal of $R$ by \cite[Prop. 1 (ii), pg. 10]{Frohlich}.

We now prove (i).
Let $A_0$ be a Hopf $R$-order of $A$ whose index $[A_R:A_0]$ is minimal in the divisibility relations among integral ideals of $R$.
Then every other Hopf $R$-order $\mathcal{O}\subseteq A$ is contained in $A_0$, since otherwise the $R$-algebra generated by $A_0$ and $\mathcal{O}$ would be a Hopf order (see the argument in \cite[Prop. 4.1.1]{Tate_FLT}) with smaller index in $A_R$ by \cite[Prop. 1, pg. 10]{Frohlich}.
The group scheme $\mathrm{Spec}(A_0)$ is then the sought maximal extension.
The statement for the minimal extension follows from Cartier duality.

To prove (ii) notice that every automorphism $\alpha$ of $G$ is induced by a Hopf-algebra automorphism $\varphi \in \mathrm{Aut}(A)$.
The automorphism $\varphi$ restricts to an automorphism of the maximal Hopf $R$-order $A_0$ of $A$, since $\varphi(A_0)$ and $\varphi^{-1}(A_0)$ must both be Hopf $R$-orders, hence contained in $A_0$.
We deduce that $\alpha$ extends to an automorphism of $G^+$ and the analogous statement for $G^-$ follows from Cartier duality.
The uniqueness of the extension follows from flatness, see Proposition \ref{prop:unique_extension_of_homs}.
\end{proof}

\begin{remark}
    Automorphisms need not extend to arbitrary prolongations.
    For example, $(\underline{\ZZ/2\ZZ} \times \mu_2) /\ZZ$ is a prolongation of $(\underline{\ZZ/2\ZZ})^2  /\QQ$ to $\ZZ$.
    However $(\underline{\ZZ/2\ZZ} \times \mu_2) /\ZZ$ has subgroup schemes of order two which are not isomorphic and thus its automorphism group is smaller than $\GL_2(\FF_2)$, whereas the automorphism group of $(\underline{\ZZ/2\ZZ})^2 /\QQ$ is isomorphic to $\GL_2(\FF_2)$.

    Note that $(\underline{\ZZ/2\ZZ})^2 /\ZZ$ is the maximal prolongation of $(\underline{\ZZ/2\ZZ})^2  /\QQ$ to $\ZZ$, and by Cartier duality $\mu_2^2 /\ZZ$ is the minimal prolongation. 
\end{remark}

We now focus on the case where $R=\mathbb{Z}_\ell$ for a prime $\ell$ and $G$ is a finite flat group scheme over $\mathbb{Q}_\ell$ annihilated by $\ell$ admitting at least one finite flat prolongation $\mathbf{G}$ over $\mathbb{Z}_\ell$.
If $\ell>2$, the group scheme $\mathbf{G}$ is, up to isomorphism, the unique prolongation of $G$ by \cite[Thm. 3.3.3] {Raynaud_type_p}.
On the other hand, if $\ell=2$ this may not be the case, as one can see be considering $\mu_2$ and $\underline{\ZZ/2\ZZ}$ over $\ZZ_2$ where they are not isomorphic and over $\QQ_2$ where they become isomorphic.
Nonetheless, we now present a condition due to Raynaud which also guarantees the uniqueness of a prolongation from $\QQ_2$ to $\ZZ_2$.

Since $\mathbb{Z}_\ell$ is a complete noetherian local ring, we have a connected-étale exact sequence \cite[pg. 138 (I)]{Tate_FLT}
\[
0 \to \mathbf{G}^0 \to \mathbf{G}\to \mathbf{G}_{\mathrm{\acute{e}t}} \to 0
\]
where $\mathbf{G}^0$ is the connected component of $\mathbf{G}$ and $\mathbf{G}_{\mathrm{\acute{e}t}}$ is its maximal étale quotient.
The \textit{biconnected} component of $\mathbf{G}$ is defined as $\mathbf{G}_{\mathrm{bi}}\coloneqq \mathbf{G}^0/\mathbf{G}_{\mathrm{mul}}$, where $\mathbf{G}_{\mathrm{mul}}$ denotes the biggest subgroup of multiplicative type (\textit{i.e.} admitting an étale Cartier dual) of $\mathbf{G}^0$.

It follows from \cite[Prop. 3.3.2 (3)]{Raynaud_type_p} (see also Section 3.3.5 in the same article) that if $\mathbf{G}$ and $\mathbf{G}'$ are two prolongations of $G$ to $\mathbb{Z}_\ell$ then there is a natural isomorphism $\mathbf{G}_{\mathrm{bi}} \cong \mathbf{G}_{\mathrm{bi}}'$.
This readily implies the following useful result.

\begin{proposition} \label{prop:unique_biconnected}
    Let $\ell$ be a prime number and $G$ a finite flat group scheme over $\mathbb{Q}_\ell$ annihilated by $\ell$, admitting at least one prolongation over $\mathbb{Z}_\ell$ which is connected with connected dual.
    Then this is the only prolongation of $G$ to $\mathbb{Z}_\ell$ up to isomorphism.
\end{proposition}

\begin{proof}
    A group scheme over $\mathbb{Z}_\ell$ that is connected and has connected dual is equal to its own biconnected component. By the discussion above, any two prolongations have isomorphic biconnected components. The result then follows from order considerations.
\end{proof}

\subsubsection{Group schemes of prime order in $\mathcal{C}_{N, \ell}$}
\label{sec:group_schemes_prime_order}
Let us put some of the theory we have developed so far into practice and classify the group schemes of order $\ell$ in $\mathcal{C}_{N, \ell}$.

\begin{proposition}
\label{prop:grp_schms_prime_order}
    Let $G$ be an object in $\mathcal{C}_{N, \ell}$ of order $\ell$.
    Then $G \cong \ZnZconst{\ell}$ or $\mu_\ell$.
\end{proposition}

\begin{proof}
    Let $p|N$ be a prime.
    We begin by showing $G(\overline{\QQ})$ is unramified at $p$.
    Let $\sigma \in I_p$ and $H \coloneqq (\sigma -1)(G(\overline{\QQ}))$.
    As $(\sigma-1)^2=0$, we see $\sigma$ fixes $H$.
    As $G(\overline{\QQ})$ is irreducible, either $H=0$ or $H=G(\overline{\QQ})$.
    In either case, we find that $\sigma$ fixes $G(\overline{\QQ})$.
    Thus $G(\overline{\QQ})$ is unramified outside of $\ell$.
    In particular, there exists a unique étale prolongation $G'$ of $G_\QQ$ to $\ZZ[\frac{1}{\ell}]$ \cite[pg. 137]{Tate_FLT}.
    Moreover, any prolongation of $G_\QQ$ to $\ZZ[\frac{1}{\ell}]$ is étale \cite[pg. 138, (II)]{Tate_FLT}.
    Thus $G'$ is the unique prolongation of $G_\QQ$ to $\ZZ[\frac{1}{\ell}]$.

    The action of $\Gamma_\QQ$ on $G(\overline{\QQ})$ furnishes us with a character $\Gamma_\QQ \rightarrow \FF_\ell^*$.
    Thus the Kronecker-Weber Theorem implies $\QQ(G) \subseteq \QQ(\zeta_\ell)$.
    In particular, either $\QQ(G)=\QQ$ or $\QQ(G)$ is totally ramified at $\ell$.
    Thus by \cite[Cor. 3.4.4]{Raynaud_type_p} we have either $G_{\ZZ_\ell} \cong \ZnZconst{\ell}$ or $\mu_\ell$.
    In the case that $G_{\ZZ_\ell} \cong \ZnZconst{\ell}$, we have $\QQ(G) = \QQ$ and thus $G' \cong \ZnZconst{\ell}$ by étaleness.
    Suppose instead that $G_{\ZZ_\ell} \cong \mu_{\ell}$.
    The Galois module $G'(\overline{\QQ})$ and hence $G'$, by étaleness, is determined by the action of inertia at $\ell$, since any extension of $\QQ$ is generated by its inertia groups.
    Thus to determine $G'$ is suffices to determine $G'_{\QQ_\ell}$.
    However, as $G'$ is a prolongation of $G_\QQ$, we have $G'_{\QQ_\ell} \cong G_{\QQ_\ell} = (G_{\ZZ_\ell})_{\QQ_\ell} \cong \mu_\ell$.
    Thus $G' \cong \mu_\ell$.

    The above argument shows, in either case, the existence of an isomorphism $\theta$ making the triple $(G_{\ZZ_\ell}, G', \theta)$ into an object of the category $\mathcal{D}_{1,\ell}$ defined in Proposition \ref{prop:equivalence_categories_prolongations}.
    Using that $G_{\QQ_\ell}$ is étale, we see that $\Aut(G_{\QQ_\ell}) \cong \FF_\ell^*$ with the action of $a \in \FF_\ell^*$ on $G_{\QQ_\ell}$ being multiplication by $a$.
    Abusing notation and letting $a \in \FF_\ell^*$ also denote multiplication by $a$ on $G_{\QQ_\ell}$, we find that $(G_{\ZZ_\ell}, G', a\theta)$ is also an object in the category $\mathcal{D}_{1,\ell}$.

    We shall now show that for any $a \in \FF_\ell^*$, the objects $(G_{\ZZ_\ell}, G', \theta)$ and $(G_{\ZZ_\ell}, G', a\theta)$ are isomorphic.
    As $G'$ is étale, we see that $\Aut(G') \cong \FF_\ell^*$ with the action of $a \in \FF_\ell^*$ on $G'$ being multiplication by $a$.
    Thus we see the pair $(\mathrm{id}_{G_{\ZZ_\ell}},a \cdot \mathrm{id}_{G'})$ defines a morphism from $(G_{\ZZ_\ell}, G', \theta)$ to $(G_{\ZZ_\ell}, G', a\theta)$ with inverse $(\mathrm{id}_{G_{\ZZ_\ell}},a^{-1} \cdot \mathrm{id}_{G'})$.
    
    As the above exhaust all possible objects $G$ can give rise to in $\mathcal{D}_\ell$, it follows from Proposition \ref{prop:equivalence_categories_prolongations} that either $G \cong \ZnZconst{\ell}$ or $\mu_\ell$.
\end{proof}

Let us underline the case $\ell =2$, in which case the category $\mathcal{C}_N$ may be replaced with the category of all finite flat group schemes defined over $\ZZ[\frac{1}{N}]$.

\begin{proposition}
\label{prop:grp_schms_order_2}
    Let $N$ be an odd integer.
    Then any finite flat group scheme over $\ZZ[\frac{1}{N}] $ of order $2$ is isomorphic to either $\ZnZconst{2}$ or $\mu_2$. 
\end{proposition}

\begin{proof}
    This can be proved via an elementary calculation using Hopf algebras.
    It is set as an exercise in \cite[pg. 19, Ex. 11]{Waterhouse_group_schemes_book} and also carried out in \cite[pg. 133, (3.2)]{Tate_FLT}.
    
    Alternatively, we may deduce the result from Proposition \ref{prop:grp_schms_prime_order}.
    Indeed, suppose $G/\ZZ[\frac{1}{N}]$ is a finite flat group scheme of order 2.
    Then $G(\overline{\QQ}) \cong \FF_2$ as abstract groups, thus the action of $\Gamma_\QQ$ on $G(\overline{\QQ})$ factors through $\FF_2^*$, so it acts trivially.
    Hence $G$ is an object of $\mathcal{C}_{N}$.
\end{proof}

Let us note that Proposition \ref{prop:grp_schms_prime_order} does not generalise to this larger category for arbitrary primes $\ell$ as the following example shows.

\begin{example}
    Let $E$ denote the elliptic curve defined by $y^2+y=x^3-3x-5$, which has conductor $3^2\cdot 11$ and thus is not semistable at 3.
    The point $P = (1,3\zeta_3+1) \in E(\overline{\QQ})$ has order 5.
    As $\overline{\zeta}_3 = -(1+ \zeta_3)$, we see $P$ is Galois conjugate to $-P = (1, -3\zeta_3-2)$.
    In particular, the subgroup $\langle P\rangle $ is preserved by the étale fundamental group of $\ZZ[\frac{1}{3}]$ and hence corresponds to a finite flat group scheme $G/\ZZ[\frac{1}{3}]$ of order 5.
    Since $\QQ(G) = \QQ(\zeta_3)$, it is clear that $G$ is not isomorphic to either of $\ZnZconst{5}$ or $\mu_5$.
    
    This is in accordance with Proposition \ref{prop:grp_schms_prime_order}, since for $\sigma \in \Gamma_\QQ \setminus \Gamma_{\QQ(\zeta_3)}$, we have  $(\sigma-1)^2(P)=4P\neq 0$ and thus $G$ does not belong to $\mathcal{C}_{3,5}$.
\end{example}

In fact, we have the more general counterexample:

\begin{example}
    Let $\ell$ be an odd prime and $d$ a squarefree integer.
    The discriminant of $\QQ(\sqrt{d})$ equals $d$ if $d \equiv 1 \mod{4}$ and $4d$ otherwise.
    Thus the Galois group $\Gamma  \coloneqq \Gal(\QQ(\sqrt{d})/\QQ)$ is a quotient of the étale fundamental group of $\ZZ[\frac{1}{2d}]$.
    
    The abelian group $\ZZ/\ell\ZZ$ has automorphism group $(\ZZ/\ell \ZZ)^*$.
    In particular via the isomorphism of abstract groups $\Gamma \cong \{\pm 1\} \leq (\ZZ/\ell\ZZ)^*$, we obtain a non-trivial action of the étale fundamental group of $\ZZ[\frac{1}{2d}]$ on $\ZZ/\ell\ZZ$.
    Thus by the equivalence of categories, there exists an étale group scheme of order $\ell$ defined over $\ZZ[\frac{1}{2d}]$ with field of points $\QQ(\sqrt{d})$.
    
    As before, this is in accordance with Proposition \ref{prop:grp_schms_prime_order}, since for $\sigma \in \Gamma_\QQ \setminus \Gamma_{\QQ(\sqrt{d})}$ and $P$ a non-zero point of the group scheme $G_\ell$ constructed above, we have  $(\sigma-1)^2(P)=4P\neq 0$ and thus $G_\ell$ does not belong to $\mathcal{C}_{2d,\ell}$.
\end{example}

For further discussions on group schemes of prime order, we refer the reader to \cite{Oort_Tate}.

\subsection{Step 3: Extensions}
\label{subsection:extensions}
In this section we build up some general background on classifying extension of group schemes by one another.
Extensions are best understood through the tools of homological algebra.
However, as seen in Example \ref{example:closedbutnotflatsubgroupscheme}, the category $\mathcal{C}_p$ is not abelian.
Thus we cannot a priori use results from homological algebra in our setting.

To remedy this, we shall first realise $\mathcal{C}_p$ as a full subcategory of a certain abelian category $\mathcal{D}$.
We then show any extension in $\mathcal{D}$ of objects coming from $\mathcal{C}_p$ in fact belongs to the essential image of $\mathcal{C}_p$.
This strategy works more generally for the category of all finite flat group schemes in place of $\mathcal{C}_p$.
Hence, we shall work directly with this category instead, and then easily deduce the results relevant to our setting.

Let $S$ be a locally noetherian scheme.
Let $(\mathrm{Sch}/S)$ be the category of schemes over $S$, $(\mathrm{FFGS}/S)$ the category of finite flat group schemes over $S$ and $(\mathrm{Ab})$ the category of abelian groups.

We equip $(\mathrm{Sch}/S)$ with the fppf topology \cite[Example 2.32]{Vistoli_FGA}, that is, the Grothen\-dieck topology  whose coverings $\{U_i \rightarrow U\}$ consist of a jointly surjective family of flat morphisms locally of finite presentation.
We denote by $\mathrm{AbSh(fppf}/S)$ the category of abelian sheaves on $(\mathrm{Sch}/S)$ with the fppf topology.

Let $G$ be a finite flat group scheme over $S$.
Define the contravariant functor\footnote{We shall ignore set-theoretic issues.
These can be dealt with using universes.
The concerned reader may consult \cite{SGA4_tome_1}.}
\[h^G \colon (\mathrm{Sch}/S) \rightarrow (\mathrm{Ab}), \hspace{0.7cm} U \mapsto \Hom_S(U,G) = G(U).\]
The functor $h^G$ is a sheaf for the fppf topology \cite[Thm. 2.55]{Vistoli_FGA}.
The above association in turn gives rise to a functor \[h \colon (\mathrm{FFGS}/S) \rightarrow \mathrm{AbSh(fppf}/S) \text{ given  by  }  G \mapsto h^G.\]
Yoneda's Lemma implies $h$ is fully faithful, thus realising $(\mathrm{FFGS}/S) $ as a full subcategory of $\mathrm{AbSh(fppf}/S)$ (see \cite[pg. 95]{BLR_book_1990} for more details).
In particular $h^G \cong h^H$ if and only if $G \cong H$.

The category $\mathrm{AbSh}(\mathrm{fppf}/S)$ is an abelian category \cite[Thm. I.3.2.1]{Tamme_etale_cohomology}.
In particular, $\mathrm{AbSh}(\mathrm{fppf}/S)$ has kernels and cokernels, and one can define the notion of a short exact sequence as in any other abelian category.
In the specific case of abelian sheaves, this definition reduces to the following one.

\begin{definition}\label{def:exact_sequence_sheaves}
    Let $\mathcal{F},\mathcal{G}$ and $\mathcal{H}$ be objects of $\mathrm{AbSh}(\mathrm{fppf}/S)$. The sequence
    \[
    0 \to \mathcal{F} \xrightarrow[]{\iota} \mathcal{G} \xrightarrow[]{p} \mathcal{H} \to 0
    \]
    is said to be \textit{exact} if the following holds:
    \begin{enumerate}
        \item The morphism $\iota$ induces a sheaf isomorphism $\mathcal{F} \cong \ker(p)$;
        \item The morphism $p$ induces a surjective map of sheaves, i.e., for every object $U$ in $(\mathrm{Sch}/S)$ and every $h\in \mathcal{H}(U)$ there exists an fppf-covering $\{V_i \to U\}$ and elements $g_i \in \mathcal{G}(V_i)$ for all $i$ such that $p_{V_i}(g_i)=h|_{V_i}\in \mathcal{H}(V_i)$.
    \end{enumerate}
\end{definition}

An exact sequence of finite flat group schemes as in Definition \ref{def:exact_sequence_group_schemes} gives rise to an exact sequence of abelian sheaves, via $h$, as the following lemma shows.

\begin{lemma}
    Let 
    \[
    0\to F \xrightarrow[]{\iota} G \xrightarrow[]{p} H \to 0
    \]
    be an exact sequence of finite and flat group schemes over $S$. Then the associated sequence of abelian sheaves
    \[
    0 \to h^F \to h^G \to h^H \to 0
    \]
    is exact in $\mathrm{AbSh}(\mathrm{fppf}/S)$.
\end{lemma}

\begin{proof}
    Condition (1) in Definition \ref{def:exact_sequence_sheaves} readily follows from condition (1) in Definition \ref{def:exact_sequence_group_schemes}.
    As for condition (2), let $U$ be an $S$-scheme and let $f\in h^H(U)=\mathrm{Hom}(U,H)$.
    Using $f$, we take the base change of $p \colon G \rightarrow H$ to $p_U \colon U\times_H G \rightarrow U \times _H H = U$, giving us an fppf-covering $\{U\times_H G \xrightarrow[]{p_U} U\}$.
    
.
    
    Taking the pullback of $f$ provides us with an element $f\circ p_U \in h^H(U\times_H G)$.
    On the other hand, the commutativity of the diagram satisfied by the fibre product implies $f \circ p_U = p \circ \pi_G$, where $\pi_G$ is the canonical morphism $U\times_H G \to G$.
    In particular, the image of $\pi_G \in h^G(U \times _H G)$ under the map $h^G \rightarrow h^H $ induced by $p$ maps to $f\circ p_U = f|_{h^H(U\times_H G)} \in h^H(U\times_H G)$.
    This shows that  condition (2) in Definition \ref{def:exact_sequence_sheaves} also holds for our sequence of sheaves which is then exact, as we wanted to show.
\end{proof}

The category $\mathrm{AbSh}(\mathrm{fppf}/S)$ has enough injectives \cite[Cor. I.3.2.2]{Tamme_etale_cohomology} thus for $n\in \mathbb{N}$ one can define the usual (group) bifunctors $\mathrm{Ext}^n(-,-)$ by applying the theory of right derived functors to the bifunctor $\mathrm{Hom}(-,-)$.
We will be mainly interested in the case $n=1$ in which, for every pair of objects $\mathcal{F}, \mathcal{H} \in \mathrm{AbSh}(\mathrm{fppf}/S)$, the elements of $\mathrm{Ext}^1(\mathcal{H},\mathcal{F})$ can be explicitly described as equivalence classes of exact sequences 
\[
0 \to \mathcal{F} \to \mathcal{G} \to \mathcal{H} \to 0
\]
under the natural equivalence relation (see \cite[\href{https://stacks.math.columbia.edu/tag/010I}{Tag 010I}]{stacks-project}).
If both $\mathcal{F}$ and $\mathcal{H}$ are representable respectively by finite flat group schemes $F$ and $H$, we will simply write $\mathrm{Ext}^1(H,F)$ in place of $\mathrm{Ext}^1(h^H,h^F)$.
In fact, in order to emphasise the dependence on the base, we shall often write $\Ext^1_S(H,F)$ for finite flat group schemes $H$ and $F$ over $S$.

This notation suggests that extensions of finite flat group schemes can be described not only as exact sequences of abelian sheaves but also as exact sequences of group schemes themselves.
This is indeed true, as is proved in the following proposition.

\begin{proposition}
    Let $H,F$ be group schemes over $S$ with $F$ affine over $S$.
    Suppose
    \[0  \rightarrow h^F \rightarrow \mathcal{G}  \xrightarrow{\psi} h^H \rightarrow 0\]
    is an exact sequence of abelian sheaves for the fppf topology over $S$.
    Then $\mathcal{G}$ is representable.

    Moreover, if $H,F$ are finite and flat over $S$, then so is $G$, the scheme representing $\mathcal{G}$.
    Furthermore, the above exact sequence of sheaves induces an exact sequence of group schemes
    \[0 \rightarrow F \rightarrow G \rightarrow H \rightarrow 0.  \]
\end{proposition}

\begin{proof}
    As $\mathcal{G} \rightarrow h^H$ is surjective, there exists a fppf cover $\{X_i \rightarrow H\}$ such that for each $i$, $\mathrm{id}_H|_{X_i} \in h^H(X_i)$ has a preimage in $\mathcal{G}_H(X_i)$.
    Let $X$ be the $S$-scheme obtained from gluing the disjoint union of the $X_i$.
    Thus there exists $y \in \mathcal{G}(X)$ mapping to the $S$-scheme morphism $ (X \xrightarrow{x} H )\in h^H(X)$.

    The fibre product $h^F \times _{h^S}h^X$ is represented by $F \times _S X$ \cite[Section I.1.2]{SGA3_tome_I}.
    We thus have $h^X$-sheaves $h^{F \times _S X} \xrightarrow{p} h^X$ and $\mathcal{G}\times _{h^H}h^{X} \xrightarrow{\psi \times h^X} h^X$.
    By \cite[Prop. I.1.4.1]{SGA3_tome_I} there is an explicit equivalence of categories between functors over $h^X$ and functors on $(\mathrm{Sch}/X)$.
    We write $\alpha_X(p)$ and $\alpha_X(\psi)$ for the functors on $(\mathrm{Sch}/X)$ corresponding to $p$ and $\psi \times h^X$ respectively.
    Likewise, we write $\alpha_H(\psi)$ for the functor on $(\mathrm{Sch}/H)$ corresponding to $\mathcal{G} \xrightarrow{\psi} h^H$.

    Let $T \xrightarrow{f} X$ be an $X$-scheme.
    Evaluating the above functors on $f$ as described in \cite[Section I.1.4]{SGA3_tome_I}, we obtain $\alpha_X(p)(f) = F(T) \times \{f\}$ and $\alpha_X(\psi)(f) = \{(g,f) \in \mathcal{G}(T) \times \{f\} | \psi(g) = x \circ f\}$.
    For any $k \in F(T)$, we have $\psi(kg) = \psi(g)$.
    In particular, $\alpha_X(p)(f)$ acts on $\alpha_X(\psi)(f)$ via $(k,f) \cdot (g,f) = (kg,f)$.
    The condition $\psi(g) = x \circ f$ ensures the first component of any two elements in $\alpha_X(\psi)(f)$ belong to the same $F(T)$ coset of $G(T)$.
    Thus the above action is simply transitive.

    As $(y, \mathrm{id}_X) \in \alpha_X(\psi)(\mathrm{id}_X)$, it follows that for any $X$-scheme $T \xrightarrow{f} X$, we have $(\mathcal{G}(f)(y), f) \in \alpha_X(\psi)(f)$.
    Likewise, as $\mathcal{G}(T) \neq \emptyset$ is a group it contains an identity element, which lies in the image of $F(T) \hookrightarrow \mathcal{G}(T)$, so $\alpha_X(p)(f) \neq \emptyset$.
    Thus we obtain a natural isomorphism between $\alpha_X(p)$ and $\alpha_X(\psi)$ by defining for an $X$-scheme $T \xrightarrow{f} X$ the map $\alpha_X(p)(f) \rightarrow \alpha_X(\psi)(f)$ given by $(k,f) \mapsto (k \cdot \mathcal{G}(f)(y), f)$.

    Being the base change of $\mathcal{G} \rightarrow h^H$,
    the $h^X$-sheaf $\mathcal{G} \times_{h^H} h^X \rightarrow h^X$ comes with canonical descent data for $h^X \rightarrow h^H$.
    Combining  \cite[\href{https://stacks.math.columbia.edu/tag/02W5}{Tag 02W5}]{stacks-project} and \cite[Thm. 14.72]{Goertz_Wedhorn_book}, we see that to prove $\alpha_H(\psi)$ is representable it suffices to observe $F \times _S X \rightarrow X$ is affine which in turn follows from \cite[Prop. 12.3 (2)]{Goertz_Wedhorn_book} since $F \rightarrow S$ is affine.
    It follows from \cite[Prop. 1.5.1]{SGA3_tome_I} that $\mathcal{G}$ is representable by some $S$-scheme $G$.

    We now show that if $H$ and $F$ are both finite and flat over $S$, then so is $G \rightarrow H \rightarrow S$.
    As $H \rightarrow S$ is a finite flat morphism and compositions of finite flat morphisms remain finite and flat  \cite[pg. 583]{Goertz_Wedhorn_book}; it is enough to show $G \rightarrow H$ is finite and flat.

    To prove $G \rightarrow H$ is a finite flat morphism, it suffices by \cite[Prop. 14.53]{Goertz_Wedhorn_book} to show $G_X \coloneqq G \times_H X \rightarrow X$ satisfies these properties.
    In fact, since $G_X$ is isomorphic to $F \times X$ as an $X$-scheme, it is enough to show $F \times X \rightarrow X$ satisfies these properties.
    These properties are preserved under base change \cite[pg. 583]{Goertz_Wedhorn_book}, thus they hold for $F \times X \rightarrow X$ since they hold for $F \rightarrow S$.

    By the Yoneda embedding, the homomorphisms in the exact sequence of sheaves induces homomorphisms of group schemes
    \[0 \rightarrow F \rightarrow G \rightarrow H \rightarrow 0.\]
    By Remark \ref{remark:sheaf_injectivity_same_as_closed_embedding}, it suffices to prove $G \rightarrow H$ is faithfully flat.
    To see this, we begin by noting that as $F$ is an $S$-group scheme, there exists morphisms $S \rightarrow F$ and $F \rightarrow S$ whose composition $S \rightarrow F \rightarrow S$ is the identity.
    In particular, $F \rightarrow S$ is surjective and thus faithfully flat since $F \rightarrow S$ is flat by assumption.
    Thus the base change $F \times X \rightarrow X$ is faithfully flat \cite[pg. 583]{Goertz_Wedhorn_book}.
    The base change of $G \rightarrow H$ by $X \rightarrow H$ is $G_X  \rightarrow X$.
    As shown above, $G_X \cong F \times X$, thus $G_X \rightarrow X$ is faithfully flat.
    Since $X \rightarrow H$ is fppf, we deduce $G \rightarrow H$ is faithfully flat by descent \cite[pg. 583]{Goertz_Wedhorn_book}.
\end{proof}

To summarise, we may treat $\Ext^1_S(H,F)$ as the set classifying extensions of $H$ by $F$ in $(\mathrm{FFGS}/S)$ up to equivalence, which comes with an abelian group law induced by the Baer sum in the category of abelian fppf sheaves over $S$.
Note that if $H$ and $F$ have $\ell$-power order then so does any extension of $H$ by $F$ \cite[pg. 136]{Tate_FLT}.

Suppose $\mathcal{C}$ is a full subcategory of $(\mathrm{FFGS}/S)$ such that, when viewed as a full subcategory of $\mathrm{AbSh}(\mathrm{fppf}/S)$, the Baer sum maps objects in $\mathcal{C}$ to objects in its essential image.
Then for objects $H,F$ of $\mathcal{C}$, we shall write $\Ext^1_\mathcal{C}(H,F)$ for the subgroup of $\Ext^1_S(H,F)$ given by classes of extensions which lie in the essential image of $\mathcal{C}$.

For example, $\mathcal{C}_{p,\ell}$ is a full subcategory of $(\mathrm{FFGS}/\ZZ[\frac{1}{p}])$ satisfying this property, since it is closed under taking products, subobjects and quotients.

\subsubsection{Examples of extensions}
\label{subsubsection:examples_of_Extensions}
 Let $n \neq 0$ be an integer.

\begin{example}[Katz-Mazur group schemes]
\label{example:Katz-Mazur_grp_schms}
The following example comes from \cite[Section 8.7, pg. 251]{Katz_Mazur}.
Let $R$ be a noetherian domain and suppose $n>0$.
Denote the field of fractions of $R$ by $K$.
Let $\varepsilon \in R^*$.
Define
\[R_\varepsilon[n] = \bigoplus_{i=0}^{n-1}R[x]/\langle x^n - \varepsilon^i\rangle.\]
This defines a finite flat group scheme $G_\varepsilon[n]$ over $R$ of order $n^2$ which is killed by $n$ and satisfies $K(G_\varepsilon[n]) = K(\zeta_n, \sqrt[n]{\varepsilon})$.
Indeed, for $S$ a subring of $\overline{K}$, the set $G_\varepsilon[n](S)$ is given by $\{(x,i)| x \in S, i \in \ZZ/n\ZZ, \text{ and } x^n = \varepsilon^i\}$.
The group law is given by \[(x,i) + (y,j) = \begin{cases}
    (xy, i+j) &\text{ if } i+j <n\\
    (xy/a, i+j-n) &\text{ if } i+j \geq n.
\end{cases}\]
The natural injection $\bigoplus^{n-1}_{i=0} R \hookrightarrow R_\varepsilon[n]$ and the quotient map $R_\varepsilon[n] \rightarrow R[x]/\langle x^n-1 \rangle $ induce an exact sequence of group schemes over $R$:
\[1 \rightarrow \mu_n \rightarrow G_\varepsilon[n] \rightarrow \ZnZconst{n} \rightarrow 0.\]
\end{example}

\begin{example}
\label{example:explicit_mu_2_by_Zmod2Z}
     Let $R$ denote the ring $\ZZ[\frac{1}{8n+1}]$.
 The ring $A \coloneqq R[x,y] / \langle y^2+2y, x^2-x+ny \rangle$ becomes a Hopf algebra with comultiplication map induced by
 \[ (x,y) + (w,z) =  (x+w - 2xw + \frac{nyz}{8n+1}(1-2x)(1-2w), \, y+z+yz  ).\]
This is clearly commutative, since it is symmetric.
Furthermore one can easily verify that $(x,y) + (0,0) = (x,y)$ and $(x,y) + (x,y) = (0,0)$. 
Thus the resulting group scheme over $R$ is annihilated by 2.
Its field of points equals $\QQ(\sqrt{8n+1})$.
We shall denote this group scheme by $(\ZnZconst{2} \,.\,  \mu_2)^\chi$, where $\chi$ is the non-trivial character of $\QQ(\sqrt{8n+1})$.
The injection $R[y] / \langle y^2+2y\rangle \hookrightarrow  A$ and quotient map $A \rightarrow A/\langle y \rangle$ induce an exact sequence of group schemes over $R$:
\[0 \rightarrow \ZnZconst{2} \rightarrow (\ZnZconst{2} \,.\,  \mu_2)^\chi \rightarrow \mu_2 \rightarrow 1.\]
\end{example}

The next examples show that we may equip the ring $B \coloneqq\ZZ[\frac{1}{4n+1}][x,y]/\langle y^2-y,x^2-x -ny \rangle$ with comultiplication maps so that the resulting Hopf algebras over $R' \coloneqq\ZZ[\frac{1}{4n+1}]$ are not isomorphic.

\begin{example}
We may equip $B$ with the comultiplication map 
\[(x,y)+(w,z) = (x+w - 2xw + \frac{2nyz}{4n+1}(1-2x)(1-2w), \, y+z - 2yz).\]
Again, one immediately verifies that this is commutative, $(x,y)+(0,0) = (x,y)$ and $(x,y) + (x,y) = (0,0)$.
Thus the resulting group scheme over $R'$ is annihilated by 2.
Its field of points equals $\QQ(\sqrt{4n+1})$.
We shall denote this group scheme by $(\underline{\ZZ/2\ZZ} \times \underline{\ZZ/2\ZZ})^\chi$, where $\chi$ is the non-trivial character of $\QQ(\sqrt{4n+1})$.
The injection $R'[y] / \langle y^2-y\rangle \hookrightarrow  B$ and quotient map $B \rightarrow B/\langle y \rangle$ induce an exact sequence of group schemes over $R'$:
\[0 \rightarrow\underline{\ZZ/2\ZZ}\rightarrow(\underline{\ZZ/2\ZZ} \times \underline{\ZZ/2\ZZ})^\chi\rightarrow\underline{\ZZ/2\ZZ} \rightarrow0.\]
\end{example}

\begin{example}
    We may also equip $B$ with the comultiplication map 
\[(x,y)+(w,z) =( x+w - 2xw + \frac{(2n+1)yz}{4n+1}(1-2x)(1-2w), \, y+z - 2yz).\]
Again, this is commutative and $(x,y)+(0,0) = (x,y)$.
In this case, $(x,y)+(x,y) = (y,0)$, $(y,0)+(y,0) = (0,0)$ and $(x,y) + (x+y-2xy,y) = (0,0)$.
In particular, the resulting group scheme over $R'$ is annihilated by 4, but not by 2.
Its field of points equals $\QQ(\sqrt{4n+1})$.
We shall denote this group scheme by $(\underline{\ZZ/4\ZZ})^\chi$, where $\chi$ is the non-trivial character of $\QQ(\sqrt{4n+1})$.
The injection $R'[y] / \langle y^2-y\rangle \hookrightarrow  B$ and quotient map $B \rightarrow B/\langle y \rangle$ induce an exact sequence of group schemes over $R'$:
\[0 \rightarrow\underline{\ZZ/2\ZZ}\rightarrow(\underline{\ZZ/4\ZZ})^\chi\rightarrow\underline{\ZZ/2\ZZ} \rightarrow0.\]
\end{example}

In reference to Example \ref{example:explicit_mu_2_by_Zmod2Z} we note:
\begin{proposition}
{(\cite[Corollary 4.2]{Schoof_one_bad_prime})}
\label{prop:extensions_mu_ell_by_Z_mod_ell_Z}
Let $\ell$ and $p$ be distinct primes.
We have that
\[\dim_{\FF_\ell}\Ext^1_{\ZZ[\frac{1}{p}]}(\mu_\ell, \underline{\ZZ/\ell\ZZ}) = \begin{cases}
    1, \text{ if } \frac{p^2-1}{24} \equiv 0 \mod{\ell}; \\
    0, \text{ otherwise.}
\end{cases}\]
\end{proposition}

\begin{remark}
    Note that for $\ell = 2$, the above is saying that 
    there are no non-split extensions if $p \equiv \pm 5 \mod{8}$ and there is a unique non-trivial extension of $\mu_2$ by $\underline{\ZZ/2\ZZ}$ if $p \equiv \pm 1 \mod{8}$ and no non-trivial extensions otherwise.
    As part of the examples in Section \ref{sec:FS_examples}, we shall  provide a proof of the case when $p \equiv \pm 5 \mod{8}$, see Lemma \ref{lem:running_example_extensions_mu_2_by_Z_mod_2_Z_over_1/3}.
    In the case $p \equiv \pm 1 \mod{8}$ our above examples give explicit realisations of these non-trivial extensions.

    Indeed, suppose $p \equiv \pm 1 \mod{8}$, set $p^* = (-1)^{\frac{p-1}{2}}p$ and let $\chi$ be the non-trivial character of $\Gal(\QQ(\sqrt{p^*})/\QQ)$.
    Then, as we have seen, $(\ZnZconst{2} \,.\,  \mu_2)^\chi$ is a non-trivial extension of $\mu_2$ by $\underline{\ZZ/2\ZZ}$ over $\ZZ[\frac{1}{p^*}] = \ZZ[\frac{1}{p}]$.
\end{remark}

\subsubsection{Studying extensions via their base changes}
The Mayer-Vietoris sequence below allows one to study extensions of finite flat group schemes by one another over some base ring $R$ via their base changes.
This can be an incredibly useful tool.
Indeed, we will often base change to an extension $R \rightarrow R'$ such that either the group schemes under consideration become étale, or such that $R'$ is a local henselian ring, in which case the connected-étale exact sequence holds \cite[(3.7)]{Tate_FLT}.
In these situations it is much easier to work with the associated extensions.

Let $R$ be a noetherian ring, $\ell \in R$ and define $\hat{R} \coloneqq \varprojlim  R/\ell^n R$.

\begin{theorem}{(Mayer-Vietoris sequence, \cite[Cor. 2.4]{Schoof_cyclotomic})} \label{thm:Mayer_Vietoris}
Let $G,H$ be finite flat group schemes over $R$.
Then the below is an exact sequence

\hspace*{-2.5em}{
\begin{tikzpicture}[descr/.style={fill=white,inner sep=1.5pt}]
        \matrix (m) [
            matrix of math nodes,
            row sep=1em,
            column sep=2.5em,
            text height=1.5ex, text depth=0.25ex
        ]
        { 0 &\Hom_{R}(G,H) & \Hom_{\hat{R}}(G,H) \times \Hom_{R[\frac{1}{\ell}]}(G,H) & \Hom_{\hat{R}[\frac{1}{\ell}]}(G,H) \\
            & \Ext^1_R(G,H) & \Ext^1_{\hat{R}}(G,H) \times \Ext^1_{R[\frac{1}{\ell}]}(G,H)  & \Ext^1_{\hat{R}[\frac{1}{\ell}]}(G,H) \\
        };

        \path[overlay,->, font=\scriptsize,>=latex]
        (m-1-1) edge (m-1-2)
        (m-1-2) edge (m-1-3)
        (m-1-3) edge (m-1-4)
        (m-1-4) edge[out=355,in=175] node[descr,yshift=0.3ex] {$\delta$} (m-2-2)
        (m-2-2) edge (m-2-3)
        (m-2-3) edge (m-2-4);

\end{tikzpicture}

}

where $\delta$ is as defined in \cite[Cor. 2.4]{Schoof_cyclotomic} and the rest of the maps are given either by base change, or taking the difference of base changes.
\end{theorem}

We now give a number of results which are particularly useful in combination with the above Mayer-Vietoris sequence.

\begin{proposition}
\label{prop:unique_extension_of_homs}
Let $R \hookrightarrow R'$ be an injection of rings and $G,H$ be affine schemes over $R$.
Suppose $G$ is flat over $ R$.
Then base change induces an injective map
    \[
        \Hom_{ R}(G,H) \hookrightarrow \Hom_{ R'}(G_ {R'}, H_{ R'}).
    \]
    If moreover $G $ and $H$ are group schemes, then the above map preserves group scheme homomorphisms.
\end{proposition}

\begin{proof}
    Write $G=\Spec A$ and $H=\Spec B$ where $A,B$ are $R$-algebras and $A$ is flat over $R$.
    Elements of $ \Hom_{ R}(G,H)$ correspond bijectively to homomorphisms of $R$-algebras $B \rightarrow A$.
    Thus it suffices to show the map
    \[
    \Hom_R(B,A) \to \Hom_{R'}(B\otimes_R R', A\otimes_R R'), \hspace{0.3cm} f \mapsto f \otimes 1
    \]
    is injective.
    Let $f,g\in  \Hom_{ R}(B,A)$.
    Suppose that the images of $f$ and $g$ under the above map coincide.
    In particular, for all $b \in B$, we have $f(b) \otimes 1 = g(b) \otimes 1$.
    Thus $(f-g)(b) \otimes 1 = 0$.
    As $A$ is a flat $R$-module and $R$ injects into $R'$, we find $A \rightarrow A \otimes_R R'$ is also injective.
    Thus $(f-g)(b) =0$, that is $f=g$.

    If $G$ and $H$ are group schemes and $f \colon B \rightarrow A$ is a homomorphism of Hopf algebras, then so is the induced map $f \otimes 1$ via a straightforward verification.
\end{proof}

\begin{corollary}
\label{Cor:if_n_kills_the_base_change_then_it_kills_you_too}
    Let $R \hookrightarrow R'$ be an injection of rings and $G$ be a flat affine group scheme over $R$.
    If multiplication by $n$ annihilates $G_{R'}$, then it also annihilates $G$.
\end{corollary}

The following proposition with a few extra conditions on the ground field is stated by Raynaud \cite[Cor. 3.3.6.]{Raynaud_type_p} and a proof has been provided by Tate \cite[pg. 152 Cor.]{Tate_FLT}.
However, these extra assumptions are unnecessary and prevent one from applying this particularly useful result when studying extensions of group schemes which are annihilated by 2.

\begin{proposition}
\label{prop:base_change_on_extension_groups}
    Let $R$ be a Dedekind domain of characteristic zero and $K$ its field of fractions.
    Let $G,H$ be finite flat group schemes over $R$.
    Suppose that $G$ is the only prolongation of $G_{K}$ to $R$ up to isomorphism.
    Then base change is an injective map from $\mathrm{Ext}^1_R(G,H)$ into $\mathrm{Ext}^1_K(G,H)$.
\end{proposition}

\begin{proof}
    As base change induces a group homomorphism from $\mathrm{Ext}^1_R(G,H)$ to $\mathrm{Ext}^1_K(G,H)$, it suffices to show any extension which base changes to the split extension over $K$ is also split over $R$.
    Let \[
    0 \rightarrow H \rightarrow E \xrightarrow{\varphi} G \rightarrow 0
    \]
    be such an extension.
    That is, we have a section $\iota \colon G_K \rightarrow E_K$ over $K$.
    We denote by $\iota(G_K)$ the schematic image of $\iota$.
    It is a subgroup scheme of $E_K$ isomorphic to $G_K$.

    Let $\textbf{G}$ be the schematic closure of $\iota(G_K)$ in $E$.
    It is a finite flat subgroup scheme of $E$ with generic fibre $\iota(G_K)$ by the proof of Lemma \ref{lem:correspondence_submodules_subschmes}.
    Let $\tau$ denote the composition $G \xrightarrow{\sim} \textbf{G} \subseteq E$, where the first map is an isomorphism whose existence is granted by our assumption that $G$ is the only prolongation of $G_K$ to $R$ up to isomorphism.
    The morphism $\sigma\coloneqq\varphi \circ \tau\colon G \to G$, when restricted to the generic fibre, is an automorphism of $G_K$ by construction.
    Since $G$ is the only prolongation of $G_K$ to $R$ up to isomorphism, we see that $G$ must be the maximal (and minimal) prolongation of $G_K$ and we deduce using Proposition \ref{prop:automorphisms_maximal_prolongation} that $\sigma \in \mathrm{Aut}(G)$.
    Hence the morphism $\tau \circ \sigma^{-1}$ is a section for the above exact sequence.
\end{proof}

\subsection{Step 4: Classifying the possible Tate modules} \label{subsec:classification_Tate_modules}

As mentioned in the introduction, Faltings' Isogeny Theorem reduces the problem of classifying abelian varieties up to isogeny to classifying their Tate modules at some given prime up to isomorphism.

We will begin this section by presenting some general arguments which allow one to rule out certain $\ell$-divisible groups from giving rise to the Tate module of an abelian variety.

We start with a few preliminary lemmas from group theory.

\begin{lemma}
\label{lemma:trivial_by_trivial_implies_order_divides_n}
    Let $\Gamma$ be a group and $N,Q$ be $\ZZ[\Gamma]$-modules upon which $\Gamma$ acts trivially.
    Let $M$ be an extension of $Q$ by $N$ such that $M$ is $n$-torsion for some positive integer $n$.
    Then the action of $\Gamma$ on $M$ factors through a quotient in which every element has order dividing $n$.
\end{lemma}
 
\begin{proof}
    Let $\sigma \in \Gamma$, $m \in M$.
    The image of $(\sigma -1)m$ in $Q$ equals zero.
    Thus $(\sigma -1)m \in N$, whence $(\sigma-1)^2m=0$.
    In particular, $(\sigma-1)^2$ annhilates $M$.

    Note $\sigma^n -1 = \sum^{n}_{i=1} \sigma^i - \sum^{n-1}_{i=0} \sigma^i = (\sum^{n-1}_{i=0} \sigma^i )(\sigma -1)$.
    As the image of $\sigma-1$ is a submodule of $N$, on which $\sigma$ acts trivially, we find $\sigma^n-1 = n(\sigma -1)$ on $M$.
    But $M$ is annihilated by $n$, thus $\sigma^n-1=0$ on $M$, as required.
\end{proof}

\begin{proposition}
\label{prop:field_of_points_is_a_p_power_extension}
    Let $p$ and $\ell$ be primes.
    Suppose $C$ is an étale group scheme belonging to $\mathcal{C}_{p, \ell}$ with a filtration of closed finite flat subgroup schemes such that every successive quotient is isomorphic  to $\underline{\ZZ/\ell \ZZ}$.
    Then $\QQ(C)/\QQ$ is an extension of degree a power of $\ell$, unramified outside of $p$.
\end{proposition}

\begin{proof}
    As $C$ is an étale group scheme over $\ZZ[\frac{1}{p}]$, the extension $\QQ(C)/\QQ$ is unramified outside of $p$.
    Let us now show that the degree of the extension $\QQ(C)/\QQ$ is a power of $\ell$.
    We shall prove this by induction.
    If $C= \underline{\ZZ/\ell\ZZ}$ then this is clear.
    Suppose the result holds for any proper closed finite flat subgroup scheme of $C$.
    By assumption, there is a closed finite flat subgroup scheme $B$ of $C$ such that $C/B \cong \underline{\ZZ/\ell \ZZ}$.
    In other words, we have a short exact sequence
    \[0 \rightarrow B \rightarrow C \rightarrow \underline{\ZZ/\ell\ZZ} \rightarrow 0\]
    where $C$ is annihilated by a power of $\ell$.
    We want to show $\Gal(\QQ(C)/\QQ)$ is an $\ell$-group.
    By assumption we know the quotient $\Gal(\QQ(B)/\QQ)$ is an $\ell$-group, hence it is enough to prove the normal subgroup $\Gal(\QQ(C)/\QQ(B))$ is an $\ell$-group.
    Viewing the above exact sequence as $\ZZ[\Gamma]$-modules  for $\Gamma =\Gal(\QQ(C)/\QQ(B))$, we may apply Lemma \ref{lemma:trivial_by_trivial_implies_order_divides_n} which shows $\Gamma$ is indeed an $\ell$-group.
\end{proof}

We denote the derived subgroup of $\Gamma$ by $D(\Gamma)$.

\begin{lemma}
\label{lemma:G_p-group_G/D(G)_cyclic_gives_G_cyclic}
    Let $\Gamma$ be an $\ell$-group for some prime $\ell$.
    Suppose the images of $g_1, \ldots, g_n \in \Gamma$ generate $\Gamma/D(\Gamma)$.
    Then $\Gamma$ is generated by $g_1, \ldots g_n$.
\end{lemma}

\begin{proof}
    Every maximal subgroup of an $\ell$-group is normal and has index $\ell$ \cite[Thm. 1.26, Lem. 1.23]{Isaacs_finite_grp_theory_book}.
    In particular, every maximal subgroup $M$ of $\Gamma$ contains $D(\Gamma)$, thus the Frattini subgroup $\Phi(\Gamma)$ contains $D(\Gamma)$.
    
    By assumption $g_1, \ldots g_n \in \Gamma$ generate the quotient $\Gamma/D(\Gamma)$.
    It follows that $\Gamma= \langle g_1, \ldots g_n \rangle D(\Gamma) = \langle g_1, \ldots g_n \rangle \Phi(\Gamma)$ by the above.
    We deduce that $\{g_1, \ldots g_n\} \cup \Phi(\Gamma)$ is a generating set for $\Gamma$.
    However, $\Phi(\Gamma)$ consists of non-generators for $\Gamma$.
    Thus $\Gamma=\langle g_1, \ldots g_n \rangle$.
\end{proof}

\begin{proposition}
\label{prop:filtration_has_no_Z/lZ_nor_mu_l}
    Let $\ell$ and $p$ be distinct primes.
    Suppose that for any simple object $V$ of $\mathcal{C}_{p,\ell}$ not isomorphic to $\underline{\ZZ/\ell \ZZ}$, we have
    $\Ext^1_{\mathcal{C}_{p,\ell}}(V, \underline{\ZZ/\ell\ZZ}) =0$.
    
    Let $A/\QQ$ be a semistable abelian variety with good reduction outside of $p$.
    Let $\mathcal{A}\to \Spec(\ZZ[\frac{1}{p}])$ be the Néron model of $A$ over $\ZZ[\frac{1}{p}]$.
    Then for every positive integer $n$, the group scheme $\mathcal{A}[\ell^n]$ admits a filtration of closed finite flat subgroup schemes in $\mathcal{C}_{p, \ell}$
    \[\mathcal{A}[\ell^n] = G_0 \supseteq G_1 \supseteq \ldots \supseteq G_m = 0\]
    such that the successive quotients $G_i/G_{i+1}$ are simple objects in $\mathcal{C}_{p, \ell}$ not isomorphic to $\underline{\ZZ/\ell \ZZ}$ or $\mu_\ell$.
\end{proposition}

\begin{proof}
For every positive integer $n$, the group scheme $\mathcal{A}[\ell^n]$ admits a filtration with successive quotients isomorphic to $\mathcal{A}[\ell]$. Hence, it suffices to prove the result for $\mathcal{A}[\ell]$.

    By Lemma \ref{lem:filtration_l_torsion}, there exists a filtration by closed finite flat subgroup schemes in $\mathcal{C}_{p, \ell}$:
    \[\mathcal{A}[\ell] = G_0 \supseteq G_1 \supseteq \ldots \supseteq G_m = 0\]
     such that the successive quotients $G_i/G_{i+1}$ are simple objects of $\mathcal{C}_{p, \ell}$.
    
    Furthermore, we may suppose that there exists $j,k$ such that for $i < j$, the quotients $G_{i}/G_{i+1}$ are either zero or isomorphic to $\underline{\ZZ/\ell\ZZ}$ and for $i>k$ the quotients $G_i/G_{i+1}$ are either zero or isomorphic to $\mu_\ell$.
    Indeed, suppose we have $G_i/G_{i+1} \cong V$ a simple group scheme not isomorphic to $\underline{\ZZ/\ell\ZZ}$  and $G_{i+1}/G_{i+2} \cong \underline{\ZZ/\ell\ZZ}$.
    Then we have an exact sequence of group schemes \cite[Prop. 5.2.9]{SGA3_tome_I}
    \[0 \rightarrow G_{i+1}/G_{i+2} \rightarrow G_{i}/G_{i+2} \rightarrow G_{i}/G_{i+1} \rightarrow 0 \]
    and so $G_i/G_{i+2}$ is an extension of $V$ by $\underline{\ZZ/\ell\ZZ} $.
    
    By assumption $\Ext^1_{\mathcal{C}_{p,\ell}}(V, \underline{\ZZ/\ell\ZZ})=0$, so $G_i/G_{i+2} \cong V \times \underline{\ZZ/\ell\ZZ}$ and thus we may replace $G_{i+1}$ by a group scheme $G_{i+1}'$ such that $G_i/G_{i+1}' \cong \underline{\ZZ/\ell \ZZ}$ and $G_{i+1}'/G_{i+2} \cong V$ \cite[Prop. 5.2.7]{SGA3_tome_I}.
    This proves there exists such a $j$.
    
    By Cartier duality, 
        $\Ext^1_{\mathcal{C}_{p,\ell}}(\mu_\ell, V) =0$ for any simple object $V$ of $\mathcal{C}_{p,\ell}$ not isomorphic to $\mu_\ell$.
        Repeating the above argument with $\Ext^1_{\mathcal{C}_{p,\ell}}(\mu_\ell, V) =0$ proves the existence of such a $k$.

    We deduce that if $\mathcal{A}[\ell]/G_j$ is non-zero, then it has a filtration with successive simple quotients isomorphic to $\underline{\ZZ/\ell\ZZ}$.
    We shall show that  $\mathcal{A}[\ell]/G_j$ is zero.

    Suppose $\mathcal{A}[\ell]/G_j$ is non-zero.
    We first prove by induction that for any positive integer $n$, there is an exact sequence
    \[0 \rightarrow M_n \rightarrow \mathcal{A}[\ell^n] \rightarrow C_n \rightarrow 0\]
    where $C_n$ has a filtration with at least $n$ successive quotients isomorphic to $\underline{\ZZ/\ell\ZZ}$.
    
    Multiplication by $\ell^{n-1}$ induces an exact sequence $0 \rightarrow \mathcal{A}[\ell^{n-1}] \rightarrow \mathcal{A}[\ell^n] \rightarrow \mathcal{A}[\ell] \rightarrow 0$.
    By induction, $\mathcal{A}[\ell^{n-1}]$ has a quotient $C_{n-1}$ which admits a filtration with at least $n-1$ successive quotients isomorphic to $\underline{\ZZ/\ell\ZZ}$.
    Using the same argument as before, we obtain a filtration of $\mathcal{A}[\ell^{n}]$ with at least $n$ successive quotients isomorphic to $\underline{\ZZ/\ell\ZZ}$, as required.
    
    By Proposition \ref{prop:field_of_points_is_a_p_power_extension}, the group $\Gamma \coloneqq\Gal(\QQ(C_n)/\QQ)$ has $\ell$-power order and $\QQ(C_n)$ is unramified outside of $p$.
    It follows from the Kronecker-Weber Theorem that the largest abelian subfield of $\QQ(C_n) $ is contained in $\QQ(\zeta_p)$.
    Thus $\Gamma/D(\Gamma)$ is cyclic.
    We deduce $\Gamma$ is also cyclic by Lemma \ref{lemma:G_p-group_G/D(G)_cyclic_gives_G_cyclic}, whence $\QQ(C_n) \subseteq  \QQ(\zeta_p)$.
    
    Let $\q$ be a prime of $\ZZ[\zeta_p, \frac{1}{p}]$.
    For every $n$, each abelian variety $A/(M_n)_\QQ$ is isogenous to $A$ over $\QQ$.
    In particular, each $A/(M_n)_\QQ$ has the same number of rational points as $A$ modulo $\q$.
    As $C_n$ is constant over $\QQ(\zeta_p)$, the abelian variety $A/(M_n)_\QQ$ has at least $\ell^n$ points over $\QQ(\zeta_p)$.
    Thus the same is true for the reduction of $A/(M_n)_\QQ$ modulo $\q$.
    However, the reduction of $A$ modulo $\q$ only has a finite number of rational points modulo $\q$, which contradicts the above.
    We deduce $\mathcal{A}[\ell] = G_j$.

    Recall that the Cartier dual of $\mathcal{A}[\ell]$ is isomorphic to $\mathcal{A}^t[\ell]$, the $\ell$-torsion   of the dual of $\mathcal{A}$, see \cite[Cor. 27.214]{Goertz_Wedhorn_II}.
    In particular, the Cartier dual of $\mathcal{A}[\ell]$ is the $\ell$-torsion of the Néron model over $\ZZ[\frac{1}{p}]$ of a semistable abelian variety over $\QQ$ with good reduction outside of $p$.
    This allows us to apply the above argument to the Cartier dual of $\mathcal{A}[\ell]$, thus proving that $G_{k+1}=0$. 
\end{proof}

From now on, suppose the hypotheses of Proposition \ref{prop:filtration_has_no_Z/lZ_nor_mu_l} hold for some choice of primes $p$ and $\ell$.
That is, for any simple object $V$ of $\mathcal{C}_{p,\ell}$ not isomorphic to $\underline{\ZZ/\ell \ZZ}$, we have
    $\Ext^1_{\mathcal{C}_{p,\ell}}(V, \underline{\ZZ/\ell\ZZ}) =0$.

If any simple object in $\mathcal{C}_{p,\ell}$ is isomorphic to either $\underline{\mathbb{Z}/\ell \mathbb{Z}}$ or $\mu_\ell$, then the proposition readily implies that a semistable abelian variety over $\mathbb{Q}$ with good reduction outside $p$ cannot exist. This situation occurs, for instance, in the examples of Section \ref{sec:FS_examples}.

On the other hand, should semistable abelian varieties over $\mathbb{Q}$ with good reduction outside $p$ exist, there must be at least one simple object $V$ in $\mathcal{C}_{p,\ell}$ that is not isomorphic to either $\underline{\mathbb{Z}/\ell \mathbb{Z}}$ or $\mu_\ell$.
In this case, and under suitable hypotheses, it remains possible to aim for a classification of all such abelian varieties, as we now explain.

Recall that in our typical situation we know an explicit abelian variety $A/\mathbb{Q}$ with good reduction outside of $p$ and bad semistable reduction at $p$. The final goal is to prove that this abelian variety is unique up to isogeny and powers. The most favourable situation for proving such a statement arises when, up to isomorphism, the simple object $V$ in $\mathcal{C}_{p,\ell}$ described above is unique. We suppose that this is the case in the rest of this discussion. 

Proposition \ref{prop:filtration_has_no_Z/lZ_nor_mu_l} then implies that $V$ is a closed subgroup scheme of $\mathcal{A}[\ell]$, where $\mathcal{A} \to \Spec(\ZZ[\frac{1}{p}])$ is the Néron model of $A$. Hence, $V$ is the kernel of some $\ell$-power isogeny $\pi: \mathcal{A} \to \mathcal{A}'$. Assume $\mathcal{A}'=\mathcal{A}$, so that $\pi \in \mathrm{End}_{\ZZ[\frac{1}{p}]}(\mathcal{A})$ and we can write $V=\mathcal{A}[\pi]$. This happens, for instance, when $\ell=2$ and $p\in \{11, 19\}$ where one has $V=J_0(p)[2]$. The case $p=11$ is treated in \cite{Schoof_one_bad_prime} while the case $p=19$ is presented in Subsection \ref{subsec:simple_19}. 

Let $B/\mathbb{Q}$ be an abelian variety with good reduction outside of $p$ and bad semistable reduction at $p$, and let $\mathcal{B}\to \Spec(\ZZ[\frac{1}{p}])$ be its Néron model. Then, for all $n\in \mathbb{N}$, the $\ell^n$-torsion scheme $\mathcal{B}[\ell^n]$ admits a filtration with successive quotients isomorphic to $\mathcal{A}[\pi]$ by Proposition \ref{prop:filtration_has_no_Z/lZ_nor_mu_l}. The main goal of this final part of the classification strategy is to show that the only group schemes in $\mathcal{C}_{p,\ell}$ that admit such a filtration are isomorphic to finite direct sums of group schemes of the form $\mathcal{A}[\pi^n]$ for some $n \in \mathbb{N}$.
More precisely, we want to show that for every non-trivial object $G$ in $\mathcal{C}_{p,\ell}$ admitting a sequence of subgroup schemes
\[
G=G_0 \supseteq G_1 \supseteq ... \supseteq G_m=0
\]
such that the successive quotients $G_i/G_{i+1}$ are isomorphic to $\mathcal{A}[\pi]$ there exists integers $r, m_1,...,m_r \in \mathbb{Z}_{>0}$ such that
\begin{equation} \label{eq:isomorphism_direct_sum}
    G \cong \bigoplus_{i=1}^r \mathcal{A}[\pi^{m_i}].
\end{equation}
Indeed, the existence of compatible isomorphisms  as in \eqref{eq:isomorphism_direct_sum} for $G=\mathcal{B}[\ell^n], n \in \mathbb{N}$, would determine the $\ell$-divisible group associated to $B$ up to isomorphism.

\begin{example}
    Suppose $\pi=\ell$ and $\mathcal{A}$ is an elliptic curve. Let $g=\mathrm{dim}(\mathcal{B})$. If isomorphism \eqref{eq:isomorphism_direct_sum} holds with $G=\mathcal{B}[\ell^n]$ for all $n \in \mathbb{N}$, then inspecting the $\overline{\mathbb{Q}}$-points of both sides of the isomorphism shows that $r=g$ and $m_1=...=m_n=n$. Hence, for all $n \in \mathbb{N}$ we have
    \[
    \mathcal{B}[\ell^n] \cong (\mathcal{A}[\ell^{n}])^g \cong \mathcal{A}^g[\ell^{n}].
    \]
\end{example}

A natural way to show that isomorphism \eqref{eq:isomorphism_direct_sum} holds for all objects $G$ as above is to proceed by induction on the rank of $G$.
The base case of the induction is clear.
As $G/G_1=G_0/G_1 \cong \mathcal{A}[\pi]$, the inductive hypothesis yields
\[
G_1 \cong \bigoplus_{i=1}^r \mathcal{A}[\pi^{m_i}]
\]
for some positive integers $r$ and $m_i$. Hence, $G$ sits in the exact sequence
\[
0 \to \bigoplus_{i=1}^r \mathcal{A}[\pi^{m_i}] \to G \to \mathcal{A}[\pi] \to 0
\]
that represents an element of
\[
\Ext^1_{\mathcal{C}_{p,\ell}}(\mathcal{A}[\pi],\bigoplus_{i=1}^r \mathcal{A}[\pi^{m_i}]) \cong \bigoplus_{i=1}^r \Ext^1_{\mathcal{C}_{p,\ell}}(\mathcal{A}[\pi],\mathcal{A}[\pi^{m_i}]).
\]

Clearly, the class represented by the sequence:
\[
0 \to \mathcal{A}[\pi^{m_j}] \to \mathcal{A}[\pi^{m_j+1}] \to \mathcal{A}[\pi] \to 0
\]
is non-trivial in $\Ext^1_{\mathcal{C}_{p,\ell}}(\mathcal{A}[\pi],\mathcal{A}[\pi^{m_j}])$. If this were to be the only non-trivial class in this extension group, we would have
\begin{equation} \label{eq:isomorphism_ext}  \Ext^1_{\mathcal{C}_{p,\ell}}(\mathcal{A}[\pi],\bigoplus_{i=1}^r \mathcal{A}[\pi^{m_i}]) \cong \bigoplus_{i=1}^r \mathbb{F}_2 \cdot e_i
\end{equation}
where each class $e_j$ is represented by the sequence
\[
0 \to \bigoplus_{i=1}^r \mathcal{A}[\pi^{m_i}] \to \bigoplus_{i\neq j} \mathcal{A}[\pi^{m_i}] \oplus \mathcal{A}[\pi^{m_j+1}] \to \mathcal{A}[\pi] \to 0.
\]
However, even under the assumption that isomorphism \eqref{eq:isomorphism_ext} holds, we cannot yet conclude that $G$ has the desired form.
Indeed, although the middle terms of the sequences in $e_j$ are isomorphic to group schemes of the form $\bigoplus_{i=1}^s \mathcal{A}[\pi^{n_i}]$, it is not clear a priori that the middle terms of \textit{all} sequences representing elements in $\smash{ \Ext^1_{\mathcal{C}_{p,\ell}}(\mathcal{A}[\pi],\bigoplus_{i=1}^r \mathcal{A}[\pi^{m_i}])}$ have this same form.

To better understand  the problem, let $\mathcal{D}$ be the full subcategory of $\mathcal{C}_{p,\ell}$ whose objects are of the form $\bigoplus_{i=1}^s \mathcal{A}[\pi^{n_i}]$. Then the classes $e_i$ are elements of
\begin{equation} \label{eq:inclusion_ext}
    \Ext^1_{\mathcal{D}}(\mathcal{A}[\pi],\bigoplus_{i=1}^r \mathcal{A}[\pi^{m_i}]) \subseteq \Ext^1_{\mathcal{C}_{p,\ell}}(\mathcal{A}[\pi],\bigoplus_{i=1}^r \mathcal{A}[\pi^{m_i}]).
\end{equation}
where the set on the left consists of all extension classes represented by objects in $\mathcal{D}$. A priori, for every pair of objects $A, B$ in $\mathcal{D}$ the set $\Ext^1_{\mathcal{D}}(A,B)$ is just a subset of $\Ext^1_{\mathcal{C}_{p,\ell}}(A,B)$. However, if  $\Ext^1_{\mathcal{D}}(A,B)$ were to be a group for all $A$ and $B$, then isomorphism \eqref{eq:isomorphism_ext} would imply that \eqref{eq:inclusion_ext} is actually an equality. Thus, $G$ would be an object of $\mathcal{D}$, as required.

Since, the Baer sum of two extensions $0 \to G_1 \to E \to G_2 \to 0$ and $0 \to G_1 \to E' \to G_2 \to 0$ is the kernel of a morphism $E\times E' \to G_2$ modulo a closed subgroup scheme isomorphic to $G_1$, one is led to prove that the category $\mathcal{D}$ is closed under taking kernels and cokernels.

This is, in a simplified form, the approach taken in \cite[Section 8]{Schoof_one_bad_prime}, which is  formulated there in a slightly more general setting. Schoof considers $\pi$ as an element of the endomorphism ring of the $\ell$-divisible group associated to $\mathcal{A}$, subject to additional assumptions. He then establishes conditions ensuring that $\Ext^1_{\mathcal{C}_{p,\ell}}(\mathcal{A}[\pi], \mathcal{A}[\pi^{m_j}])$ has order $2$ \cite[Lem. 8.1]{Schoof_one_bad_prime}, and that the full subcategory of $\mathcal{C}_{p,\ell}$ whose objects are of the form $\bigoplus_{i=1}^s \mathcal{A}[\pi^{n_i}]$ is closed under taking kernels and cokernels \cite[Cor. 8.2]{Schoof_one_bad_prime}. The discussion culminates in \cite[Thm. 8.3]{Schoof_one_bad_prime}, which provides a classification, under the suitable assumptions that we have partially mentioned here, of the $\ell$-divisible groups associated with abelian varieties with bad semistable reduction at $p$ only.

For the reader’s convenience, we restate the result below in a form adapted to our setting, specialising the hypotheses and notation to the level of generality needed in this paper.

\begin{theorem}{(\cite[Thm. 8.3]{Schoof_one_bad_prime})} \label{thm:Schoof_classification_p_divisible groups}
    Let $p$ and $\ell$ be primes, and $G=\{G_n\}$ an $\ell$-divisible group over $\mathbb{Z}[\frac{1}{p}]$. Suppose $R=\mathrm{End}(G)$ is a discrete valuation ring with uniformizer $\pi$ and residue field $k=R/\pi R$. Suppose furthermore that:
    \begin{itemize}
        \item Every group scheme $G_n$ is an object in the category $\mathcal{C}_{p,\ell}$;
        \item The map $\mathrm{Hom}_{\mathbb{Z}[\frac{1}{p}]}(G[\pi],G[\pi]) \to \Ext^1_{\mathcal{C}_{p,\ell}}(G[\pi],G[\pi])$ obtained by applying the $\mathrm{Hom}_{\mathbb{Z}[\frac{1}{p}]}(G[\pi],-)$ functor to the exact sequence
        \[
        0 \to G[\pi] \to G[\pi^2] \to G[\pi] \to 0
        \]
        is an isomorphism of one-dimensional $k$-vector spaces.
    \end{itemize}
    Let $H=\{H_n\}$ be an $\ell$-divisible group over $\mathbb{Z}[\frac{1}{p}]$ for which the following hold:
    \begin{itemize}
        \item Every group scheme $H_n$ is an object in the category $\mathcal{C}_{p,\ell}$;
        \item Each $H_n$ admits a filtration with flat closed subgroup schemes and successive quotients isomorphic to $G[\pi]$.
    \end{itemize}
    Then $H$ is isomorphic to $G^g$ for some $g\geq 0$.
\end{theorem}

\section{The Fontaine--Schoof strategy: introductory examples}
\label{sec:FS_examples}

We now put into practice the Fontaine-Schoof strategy illustrated in Section \ref{Section:FS_strategy} by proving that there is no semistable abelian variety over $\mathbb{Q}$ with good reduction outside exactly one of 3 or 5.

For \textbf{Step 1} of the strategy, we determine the fields of definition of points of simple objects in $\mathcal{C}_3$ and $\mathcal{C}_5$. The following two lemmas will be useful.

\begin{lemma}
\label{lemma:running_examples_points_generate_ell_group}
    Let $p,\ell$ be distinct primes and $L \subseteq \overline{\mathbb{Q}}$ be the smallest field containing all the finite extensions of $\mathbb{Q}$ satisfying properties $(a)$--$(c)$ from Lemma \ref{lem:stability_properties_field_extension}.

    Suppose there is a Galois extension $F/\QQ$ contained in $L$ such that $\Gal(L/F)$ is an $\ell$-group.
    Then every simple object in $\mathcal{C}_{p, \ell}$ has its points defined over $F$.
    In particular, any such object has order $\ell$.
\end{lemma}

\begin{proof}
    By Corollary \ref{Cor:properties_of_annihilated_by_ell_with_Fontaine}, the field of points of any simple group scheme in $\mathcal{C}_{p, \ell}$  satisfies the assumptions of Lemma \ref{lem:stability_properties_field_extension}.
    For any prime $\ell$, the action of a finite group on a simple module over a field of characteristic $\ell$ factors through the quotient by its maximal normal $\ell$-subgroup \cite[Cor. 6.2.2]{Webb}.
    Hence, $\Gal(L/F)$ acts trivially on the points of any simple object in $\mathcal{C}_{p, \ell}$.

    It follows that the points of any simple object in $\mathcal{C}_{p,\ell}$ are defined over $F$.
    Thus by étaleness \cite[pg. 137]{Tate_FLT} we have that any simple object $G$ in $\mathcal{C}_{p,\ell}$ satisfies $G_F \cong \ZnZconst{\ell}$.
    In particular $G$ has order $\ell$.
\end{proof}

The following lemma is written in much greater generality than what is needed in this section for reasons of versatility.
This formulation will be useful in Sections \ref{sec:19} and \ref{sec:simple_objects_of_order_4}.
The impatient reader looking to get on with the proofs of the main results in this section should feel free to skip it, filling in the easy gaps left behind by its absence in what follows.

\begin{lemma}
\label{lem:explicit_field_extensions_satisfying_the_bounds}
Let $p$ be an odd prime and $L \subseteq \overline{\mathbb{Q}}$ be the smallest field containing all the finite extensions of $\mathbb{Q}$ satisfying properties $(a)$--$(c)$ from Lemma \ref{lem:stability_properties_field_extension} with $\ell =2$.

    Then the following hold true:
    \begin{enumerate}
        \item The maximal abelian extension of $\QQ$ contained in $L$ is $\QQ(i, \sqrt{p})$.
        \item If $p \equiv 1 \mod{4}$ then for any $\varepsilon \in \ZZ[\frac{1+\sqrt{p}}{2}]^*$ the field $\QQ(i,\sqrt{\varepsilon})$ is Galois and contained in $L$.
        \item If $p \equiv 3 \mod{8}$, then the ray class field $F_{(2)}$ of $\QQ(\sqrt{-p})$ of modulus $(2)$ is contained in $L$.
        Moreover, if $p>3$, then $[F_{(2)}\colon\QQ(\sqrt{-p})]$ is divisible by 3.
    \end{enumerate}
\end{lemma}

\begin{proof}
    An easy check shows both $\QQ(i) $ and $\QQ(\sqrt{p})$, are contained in $L$.
    Hence it suffices to prove any other abelian extension of $\QQ$ inside $L$ is contained in $\QQ(i, \sqrt{p})$.

    Let $F/\QQ$ be an abelian extension contained in $L$.
    By assumption, $\Gamma^{2}_2=1$ and $\Gamma^1_p =1$ since $L$ is tamely ramified at $p$.
    Thus, by class field theory \cite[Thm. V.6.2, Cor. VI.5.7]{Neukirch_ANT_book}, $F$ is contained in the ray class field of modulus $4p\infty$, where we let $\infty$ denote the infinite prime of $\QQ$.
    That is, $F$ is contained in $\QQ(i, \zeta_p)$ where $\zeta_p \in \overline{\QQ}$ is a primitive $p$th root of unity \cite[pg. 366]{Neukirch_ANT_book}.
    The Galois group $\Gal(\QQ(i, \zeta_p)/\QQ)$ is isomorphic to $C_2 \times C_{p-1}$ and its inertia subgroup at $p$ is $\Gal(\QQ(i, \zeta_p)/\QQ(i)) \cong C_{p-1}$.
    Since $F \subseteq L$, the inertia groups at primes above $p$ have order at most 2, thus $F$ must be contained in $\QQ(i, \sqrt{p})$.

    We shall now prove (2).
    If $\varepsilon \in \{\pm 1\}$, then we are done, so suppose not.
    Let $m_{\varepsilon}$ be the minimal polynomial of $\varepsilon$ over $\mathbb{Q}$.
    It has degree 2 since $\varepsilon \neq \pm 1$.
    Let $\varepsilon'$ denote the other root of $m_\varepsilon$.
    The polynomial $m_\varepsilon(x^2)$ is monic with integer coefficients and has roots $\pm\sqrt{\varepsilon}, \pm\sqrt{\varepsilon'}$.
    As $\varepsilon \cdot \varepsilon' = \pm 1$, we have $\sqrt{\varepsilon}\cdot \sqrt{\varepsilon'}= \sqrt{\pm1}$.
    Hence $F = \QQ(i,\sqrt{\varepsilon})$ is the compositum of the splitting field of $m_\varepsilon(x^2)$ and $\QQ(i)$.
    In particular $F$ is Galois and the Galois group $\Gal(F/\QQ(\sqrt{p}))$ is an elementary abelian 2-group.

    As $\QQ(\sqrt{p})$ ramifies exactly at $p$, it suffices to show the extension $F/\QQ(\sqrt{p})$ is unramified outside of 2 and it satisfies property $(c)$ of Lemma \ref{lem:stability_properties_field_extension}.
    We have already seen that these properties are satisfied by the subextension $\QQ(i,\sqrt{p})/\QQ(\sqrt{p})$.
    By Example \ref{example:Katz-Mazur_grp_schms} there is a finite flat group scheme $G_{\varepsilon}[2]$ defined over $\ZZ[\frac{1+\sqrt{p}}{2}]$ which is annihilated by 2 satisfying $\QQ(\sqrt{p})(G_\varepsilon[2]) = \QQ(\sqrt{\varepsilon})$.
    In particular, as $G_\varepsilon[2]$ is étale over $\ZZ[\frac{1+\sqrt{p}}{2}, \frac{1}{2}]$ by \cite[pg. 138, (II)]{Tate_FLT}, the field extension $\QQ(\sqrt{\varepsilon})/\QQ(\sqrt{p})$ is unramified outside of 2.
    Moreover, by \cite[Thm. A]{Fontaine_pas_de_variete_abelienne} the extension $\QQ(\sqrt{\varepsilon})/\QQ(\sqrt{p})$ also satisfies the conditions on the higher ramification groups at the primes above 2.
    It follows that $F/\QQ(\sqrt{p})$ is unramified outside of $2$ and moreover applying \cite[Prop. IV.14]{Serre_book_1979} we find  $F/\QQ(\sqrt{p})$ also satisfies the conditions on the higher ramification groups at the primes above 2.
    As noted above, this completes the proof.

    We shall now prove (3).
    As $p \equiv 3 \mod{8}$, the prime 2 is inert in $\QQ(\sqrt{-p})$ and only $p$ ramifies in $\QQ(\sqrt{-p})$.
    Thus $\QQ(\sqrt{-p})$ and its ray class field of modulus 2 is contained in $L$, since it is at most tamely ramified at 2.
    Let us write $\OO = \ZZ[\frac{1+ \sqrt{-p}}{2}]$, the maximal order in $\QQ(\sqrt{-p})$.
    Writing $\mathrm{Cl}_{(2)}(F)$ and $\mathrm{Cl}(F)$ for the ray class group of modulus $(2)$ of $F=\QQ(\sqrt{-p})$ and class group of $F=\QQ(\sqrt{-p})$ respectively, we have the following exact sequence \cite[Prop. 3.2.3]{Cohen_advanced_computational_ANT}:
    \[\OO^* \rightarrow (\OO/2\OO)^* \rightarrow \mathrm{Cl}_{(2)}(F) \rightarrow\mathrm{Cl}(F) \rightarrow 1.\]
    For $p>3$, we have $\OO^*= \{\pm 1\}$.
    Using that $2$ is inert in $F$, the above exact sequence becomes
    \[1 \rightarrow C_3 \rightarrow \mathrm{Cl}_{(2)}(F) \rightarrow \mathrm{Cl}(F) \rightarrow 1,\]
    which completes the proof.
\end{proof}

We now determine the fields of points of simple objects in $\mathcal{C}_3$ and $\mathcal{C}_5$.

\begin{proposition}
\label{prop:running_example_simple_objects}
    Any simple object in $\mathcal{C}_3$ has its points defined over $\QQ$.
    In particular, any such object has order 2.
\end{proposition}

\begin{proof}
    By Corollary \ref{Cor:properties_of_annihilated_by_ell_with_Fontaine}, the field of points of any simple group scheme in $\mathcal{C}_{3}$  satisfies the assumptions of Lemma \ref{lem:stability_properties_field_extension}.
Let $L \subseteq \overline{\mathbb{Q}}$ be the smallest field containing all the finite extensions of $\mathbb{Q}$ satisfying properties $(a)$--$(c)$ from Lemma \ref{lem:stability_properties_field_extension} with $p=3$ and $\ell=2$.

By Corollary \ref{cor:bound_root_discriminant}, the root discriminant of $L$ is less than $4\sqrt{3} < 6.93$.
Diaz y Diaz's tables \cite{DiazyDiaz_minorations} show that any number field of degree at least 12 has root discriminant greater than 7.4.
Thus we deduce that $[L\colon\QQ]\leq 11$.
Lemma \ref{lem:explicit_field_extensions_satisfying_the_bounds} shows that $\QQ(i, \sqrt{3})$ is contained in $L$.
Thus either $L=\QQ(i, \sqrt{3})$ or $[L\colon \QQ]=8$.
In either case, $\Gal(L/\QQ)$ is a 2-group.
Thus the result follows from Lemma \ref{lemma:running_examples_points_generate_ell_group}.
\end{proof}

\begin{proposition}
\label{prop:running_example_simple_objects_C5}
    Any simple object in $\mathcal{C}_5$ has its points defined over $\QQ$.
    In particular, any such object has order 2.
\end{proposition}

\begin{proof}
    By Corollary \ref{Cor:properties_of_annihilated_by_ell_with_Fontaine}, the field of points of any simple group scheme in $\mathcal{C}_{5}$  satisfies the assumptions of Lemma \ref{lem:stability_properties_field_extension}.
    Let $L \subseteq \overline{\mathbb{Q}}$ be the smallest field containing all the finite extensions of $\mathbb{Q}$ satisfying properties $(a)$--$(c)$ from Lemma \ref{lem:stability_properties_field_extension} with $p=5$ and $\ell=2$.

    Let $\varepsilon = \frac{1+\sqrt{5}}{2}$, which is a fundamental unit in $\ZZ[\frac{1+\sqrt{5}}{2}]$.
    Thus Lemma \ref{lem:explicit_field_extensions_satisfying_the_bounds} shows that $F=\QQ(i, \sqrt{\varepsilon})$ is contained in $L$.
    Observe that $[F \colon \QQ]=8$.
    Indeed, if not then $\sqrt{\varepsilon} = a+bi$ for some $a, b \in \QQ(\sqrt{5})$.
    Squaring both sides we find $ab=0$.
    As $\sqrt{\varepsilon}$ is an algebraic integer, we see that if $b=0$ then $\varepsilon$ is not a fundamental unit in $\QQ(\sqrt{5})$, which gives a contradiction.
    If $a=0$, then we find that $-\varepsilon$ is not a fundamental unit in $\QQ(\sqrt{5})$ which again gives a contradiction.

    By Corollary \ref{cor:bound_root_discriminant}, the root discriminant of $L$ is less than $4\sqrt{5} < 8.95$.
    Diaz y Diaz's tables \cite{DiazyDiaz_minorations} show that any number field of degree at least $18$ has root discriminant greater than $9.3$.
    We deduce that $[L\colon\QQ]\leq 17$.
    As $[F\colon \QQ] =8$ either $L=F$ or $[L\colon \QQ]=16$.
    In either case, $\Gal(L/\QQ)$ is a 2-group.
    Thus the result follows from Lemma \ref{lemma:running_examples_points_generate_ell_group}.
\end{proof}

As the simple objects in $\mathcal{C}_3$ and $\mathcal{C}_5$ have order 2, \textbf{Step 2} of the classification strategy is taken care of by Proposition \ref{prop:grp_schms_order_2}, which shows they must be isomorphic to either $\underline{\ZZ/2\ZZ}$ or $\mu_2$.

\textbf{Step 3} follows immediately from applying Proposition \ref{prop:extensions_mu_ell_by_Z_mod_ell_Z}.
However, in order to illustrate the machinery developed in Subsection \ref{subsection:extensions} for classifying extensions, we now prove the relevant part of Proposition \ref{prop:extensions_mu_ell_by_Z_mod_ell_Z} for our needs.

\begin{lemma}
\label{lem:running_example_extensions_mu_2_by_Z_mod_2_Z_over_1/3}
Let $p$ be a prime congruent to $ \pm 5$ modulo 8.
Then $\Ext^1_{\ZZ[\frac{1}{p}]}(\mu_2, \underline{\ZZ/2\ZZ})$ is trivial.
\end{lemma}

\begin{proof}
    We shall use Theorem \ref{thm:Mayer_Vietoris} with $R = \ZZ[\frac{1}{p}]$ and $\ell =2$.
    We begin by analysing the different groups of homomorphisms.
    For $S=\ZZ[\frac{1}{p}]$ or $\ZZ_2$, we have that $\mu_2$ is connected, whereas the connected component of the identity of $\ZnZconst{2}$ is the trivial group.
    As morphisms of schemes are continuous, it follows that $\mathrm{Hom}_S(\mu_2, \ZnZconst{2})=0$.
    For $S= \ZZ[\frac{1}{2p}]$ or $\QQ_2$ we have that $(\mu_2)_S \cong (\ZnZconst{2})_S$ is étale.
    Since the étale fundamental group of $S$ acts trivially on the geometric points of $(\ZnZconst{2})_S$, we have $\mathrm{Hom}_S(\mu_2, \ZnZconst{2}) \cong \mathrm{Hom}(\FF_2, \FF_2) \cong \FF_2$.
    Thus the Mayer-Vietoris sequence reduces to \[0 \rightarrow \FF_2 \rightarrow\FF_2 \rightarrow \Ext^1_{\ZZ[\frac{1}{p}]}(\mu_2, \underline{\ZZ/2\ZZ}) \rightarrow \Ext^1_{\ZZ_2}(\mu_2, \underline{\ZZ/2\ZZ}) \times \Ext^1_{\ZZ[\frac{1}{2p}]}(\mu_2, \underline{\ZZ/2\ZZ}).\]
    Hence it suffices to show that any extension in $\Ext^1_{\ZZ[\frac{1}{p}]}(\mu_2, \underline{\ZZ/2\ZZ})$ splits when base changed to either $\ZZ_2$ or $\ZZ[\frac{1}{2p}]$.
    
    Let $G \in \Ext^1_{\ZZ[\frac{1}{p}]}(\mu_2, \underline{\ZZ/2\ZZ})$ so $G$ sits in an exact sequence
    \[0 \rightarrow \underline{\ZZ/2\ZZ} \rightarrow G \rightarrow \mu_2 \rightarrow 0 .\]
    As the connected component of $\underline{\ZZ/2\ZZ}$ over $\ZZ_2$ is trivial, we see the connected component of $G_{\ZZ_2}$ is $(\mu_2)_{\ZZ_2}$.
    Thus the connected-étale sequence \cite[(3.7)]{Tate_FLT} provides a section to the above exact sequence base changed to $\ZZ_2$.
    Hence $G_{\ZZ_2} \cong \mu_2 \times  \underline{\ZZ/2\ZZ}$, that is, the image of $G$ in $\Ext^1_{\ZZ_2}(\mu_2, \underline{\ZZ/2\ZZ})$ is trivial.
    In particular $\Gamma_{\QQ_2}$ acts trivially on $G(\overline{\QQ}_2)$.
    Since $G_{\ZZ[\frac{1}{2p}]}$ is étale the extension $\QQ(G)/\QQ$ is unramified outside of $p$.
    
    As $G_{\ZZ_2}$ is killed by 2, Corollary \ref{Cor:if_n_kills_the_base_change_then_it_kills_you_too} asserts that $G$ is also killed by 2.
    Lemma \ref{lemma:trivial_by_trivial_implies_order_divides_n} implies that the Galois group of $\QQ(G)/\QQ$ is an elementary abelian 2-group.
    Thus $\QQ(G)/\QQ$ is at most tamely ramified at $p$, so by the Kronecker-Weber Theorem either $\QQ(G)=\QQ$ or $\QQ(\sqrt{p^*})$ where $p^* = (-1)^{\frac{p-1}{2}}p$.
    But $p^* \equiv 5 \mod{8}$, so 2 is inert in $\QQ(\sqrt{p^*})$.
    Hence we must have $\QQ(G) =\QQ$ since $\Gamma_{\QQ_2}$ acts trivially on $G(\overline{\QQ}_2)$.
    Hence by étaleness \cite[pg. 137]{Tate_FLT}, $G_{\ZZ[\frac{1}{2p}]} \cong \mu_2 \times  \underline{\ZZ/2\ZZ}$, which completes the proof.
\end{proof}


    

We can now proceed to \textbf{Step 4} and conclude our classification.

\begin{theorem}
    Let $p$ be either 3 or 5.
    Then there is no semistable abelian variety over $\QQ$ with good reduction outside of $p$.
\end{theorem}

\begin{proof}
    By Propositions \ref{prop:running_example_simple_objects}, \ref{prop:running_example_simple_objects_C5} any simple object in $\mathcal{C}_p$ has order 2.
    Proposition \ref{prop:grp_schms_order_2} thus implies that any simple object in $\mathcal{C}_p$ is isomorphic to either $\mu_2$ or $\ZnZconst{2}$.
    By Lemma \ref{lem:running_example_extensions_mu_2_by_Z_mod_2_Z_over_1/3}, we have $\Ext^1_{\ZZ[\frac{1}{p}]}(\mu_2, \underline{\ZZ/2\ZZ})=0$.
    
    Suppose now for a contradiction that $A/\QQ$ is a semistable abelian variety with good reduction outside of $p$.
    Let $\mathcal{A}\to \Spec(\ZZ[\frac{1}{p}])$ be the Néron model of $A$ over $\ZZ[\frac{1}{p}]$.
    The previous paragraph allows us to apply Proposition \ref{prop:filtration_has_no_Z/lZ_nor_mu_l}.
    From which we deduce that $\mathcal{A}[2]$ has a filtration by closed finite flat group schemes in $\mathcal{C}_p$ such that the successive quotients are simple objects in $\mathcal{C}_p$ not isomorphic to $\mu_2$ or $\ZnZconst{2}$.
    However, we have already shown any simple object in $\mathcal{C}_p$ is isomorphic to $\mu_2$ or $\ZnZconst{2}$. Contradiction.
\end{proof}

The motivated reader may verify that the simple objects in $\mathcal{C}_{13}$ are isomorphic to either $\underline{\ZZ/2\ZZ}$ or $\mu_2$, thus the same argument shows there is no semistable abelian variety with good reduction outside of $13$.
\section{Semistable abelian varieties over $\QQ$ with bad reduction at 19 only} \label{sec:19}
In this section, we classify semistable abelian varieties over $\mathbb{Q}$ with bad reduction only at $19$. To do so, we once again follow the Fontaine–Schoof strategy, which is considerably more involved than in the examples discussed in Section \ref{sec:FS_examples}.

\subsection{Step 1} \label{subsec:19_step_1}
Recall that we denote the derived subgroup of a group $\Gamma$ by $D(\Gamma)$.
Setting $D^1(\Gamma) \coloneqq D(\Gamma)$ we define $D^n(\Gamma) \coloneqq D(D^{n-1}(\Gamma))$ for $n \geq 2$.

\begin{lemma}
\label{lemma:gp_thry}
Let $\Gamma$ be a finite group such that $D(\Gamma)/D^2(\Gamma) \cong C_3$ and $|D^2(\Gamma)| \leq 11$, then $D^2(\Gamma)$ is a 2-group.
\end{lemma}

\begin{proof}
Note that if we take the intersection of two normal subgroups, of any given group, which contain the derived subgroup and a $p$-Sylow subgroup ($p$ prime), then their intersection is again normal and also contains the derived subgroup and all $p$-Sylow subgroups.  
In particular, there is a minimal such subgroup, which thus must also be characteristic.

Let $\Psi$ be the minimal normal subgroup of $D^2(\Gamma)$ containing its $3$-Sylow subgroups and $D^3(\Gamma)$.
Then by the above, $\Psi$ is characteristic in $D^2(\Gamma)$, and thus normal in $\Gamma$.
The quotient $D^2(\Gamma)/\Psi$ is abelian of order not divisible by $3$.
We may thus apply \cite[Lemma 3.1]{Schoof_cyclotomic}, from which we deduce that if $D^2(\Gamma)/\Psi$ is furthermore cyclic, then $\Psi = D^2(\Gamma)$.
Checking through the groups of order at most 11, we deduce  $D^2(\Gamma)$ is either a non-cyclic $2$-group, or a (possibly trivial) 3-group.

If $D^2(\Gamma)$ is a $3$-group, then so is $D(\Gamma)$.
By Lemma \ref{lemma:G_p-group_G/D(G)_cyclic_gives_G_cyclic} we find that $D(\Gamma)$ is cyclic and thus $D(\Gamma) \cong C_3$.
In particular, $D^2(\Gamma)$ is trivial.
We thus conclude $D^2(\Gamma)$ is a 2-group.
\end{proof}

\begin{proposition}
\label{prop:19_geom_points}
Any simple group scheme in $\mathcal{C}_{19}$ has its points defined over $\QQ(\sqrt{-19}, \alpha)$, where $\alpha$ satisfies $\alpha^3 -2\alpha -2=0$.
\end{proposition}

\begin{proof}
By Corollary \ref{Cor:properties_of_annihilated_by_ell_with_Fontaine}, the field of points of any simple group scheme in $\mathcal{C}_{19}$ satisfies the assumptions of Lemma \ref{lem:stability_properties_field_extension}.
Let $L \subseteq \overline{\mathbb{Q}}$ be the smallest field containing all the finite extensions of $\mathbb{Q}$ satisfying properties $(a)$--$(c)$ from Lemma \ref{lem:stability_properties_field_extension} with $p=19$ and $\ell=2$.
Let $\Gamma \coloneqq \Gal(L/\QQ)$.
By Lemma \ref{lem:explicit_field_extensions_satisfying_the_bounds}, we have $L^{D(\Gamma)} =\QQ(i, \sqrt{-19})$.

The degree three extension of $\QQ(\sqrt{-19})$ given by adjoining a root $\alpha$ of $x^3 -2x -2=0$ is ramified only at $(2)$ and is Galois over $\QQ$.
As the only ramification in the extension $\QQ( \sqrt{-19}, \alpha)/\QQ(\sqrt{-19})$ occurs at the prime $(2)$ and is tame, we deduce $F \coloneqq \QQ(i, \sqrt{-19},\alpha)$ is contained in the ray class field of modulus $(2)$ of $\QQ(\sqrt{-19})$ and thus in $L$ by Lemma \ref{lem:explicit_field_extensions_satisfying_the_bounds}.

Since $F/L^{D(\Gamma)}$ has degree 3, and is only ramified at $(1+i)$, it is contained in the ray class field of modulus $(1+i)$.
We shall now use class field theory to show $F$ is the ray class field of $L^{D(\Gamma)}$ of modulus $(2)$.
In particular, $F$ has conductor $(1+i)$ and any other non-trivial abelian extension of $L^{D(\Gamma)}$ ramified only at $(1+i)$ has conductor at least $(1+i)^3$.

Let $\OO$ be the ring of integers of $L^{D(\Gamma)}$.
Recall that $\mathrm{Cl}_{(2)}(L^{D(\Gamma)})$, the ray class group of $L^{D(\Gamma)}$ of modulus $(2)$, fits in the following exact sequence \cite[Prop. 3.2.3]{Cohen_advanced_computational_ANT}
\[\OO^* \rightarrow \left(\OO/2\OO\right)^* \rightarrow \mathrm{Cl}_{(2)}(L^{D(\Gamma)}) \rightarrow \mathrm{Cl}(L^{D(\Gamma)}) \rightarrow 0\]
where $\mathrm{Cl}(L^{D(\Gamma)})$ denotes the class group of $L^{D(\Gamma)}$.
As $-19 \equiv 5 \mod{8}$, we see $2\OO = (i+1)^2$ and $(i+1)$ has residual degree 2.
Thus, as abelian groups $(\OO/2\OO)^* \cong C_2^2 \times C_3$.
The fundamental unit $\varepsilon$ of $L^{D(\Gamma)}$ has minimal polynomial $x^4 + 26x^3 + 338x^2 - 26x + 1$.
As the absolute norm $N(i-\varepsilon)=340$ is not divisible by $N(2)=16$, we see $i \not \equiv \varepsilon \mod{2}$.
Likewise, as $N(i-1)= 4$ and $N(\varepsilon-1)=340$, both $i$ and $\varepsilon$ do not equal $1$ modulo $2$.
In particular, the image of $\OO^*$ in $\left(\OO/2\OO\right)^*$ is a group of order 4.
The class number of $L^{D(\Gamma)}$ is trivial.
Thus it follows that $\mathrm{Cl}_2(L^{D(\Gamma)}) \cong C_3$, whence $F$ is  the ray class field extension of $L^{D(\Gamma)}$ of modulus $(2)$.

We shall now show $L^{D^2(\Gamma)}= F$.
First note that $L^{D^2(\Gamma)}/L^{D(\Gamma)}$ may only be ramified at $(1+i)$ since $L^{D(\Gamma)}$ is ramified at 19 and the ramification degree of $L$ at 19 is (at most) 2.
Suppose $K/L^{D(\Gamma)}$ is an abelian extension of $L^{D(\Gamma)}$ ramified only at $(1+i)$ containing $F$ as a proper subfield.
We shall show $K$ cannot be contained in $L$.

Since $\Gal(K/L^{D(\Gamma)})$ is of order $3[K\colon F]$, it has $3([K\colon F]-1)$ many characters which do not factor through $\Gal(F/L^{D(\Gamma)})$.
The field cut out by such a character has conductor at least $(1+i)^3$, since it is an abelian extension of $L^{D(G)}$ not contained in $F$.
Hence all such characters have Artin conductor divisible by  $(1+i)^3$ \cite[Thm. VII.11.10]{Neukirch_ANT_book}.

The Artin conductors of the non-trivial characters of $\Gal(K/L^{D(\Gamma)})$ which factor through $\Gal(F/L^{D(\Gamma)})$ are determined by their Artin conductors viewed as characters of $\Gal(F/L^{D(\Gamma)})$ \cite[Prop. VII.11.7 (ii)]{Neukirch_ANT_book}.
As $F/L^{D(\Gamma)}$ is unramified outside of $(1+i)$ and tamely ramified at $(1+i)$, the Artin conductors \cite[pg. 527]{Neukirch_ANT_book} of the non-trivial characters of $\Gal(F/L^{D(\Gamma)})$ equal $(1+i)$.

Thus by the conductor-discriminant formula \cite[pg. 534, VII.11.9]{Neukirch_ANT_book} and \cite[Prop. III.8]{Serre_book_1979}, the absolute value of the discriminant of $K$ is at least \[(2^4 \cdot 19^2)^{3[K:F]} \cdot 2^4 \cdot (4^3)^{3([K:F]-1)} = 2^{\frac{5}{2}[K:\QQ]-2} \cdot 19^{\frac{1}{2} [K:\QQ]} > 2^{2[K:\QQ]} \cdot 19^{\frac{1}{2} [K:\QQ]}\]
where the final inequality comes from the fact that $[K \colon  \QQ] \geq 24$ since $K$ properly contains $F$.
The discriminant of $K$ is thus larger than that allowed by Corollary \ref{cor:bound_root_discriminant}.
Thus $L^{D^2(\Gamma)}= F$.

By Corollary \ref{cor:bound_root_discriminant}, the root discriminant of $L$ is less than $4\sqrt{19} < 17.44$.
Diaz y Diaz's tables \cite{DiazyDiaz_minorations} in turn imply that $[L\colon\QQ]\leq 137$.
Since $[ L^{D^2(\Gamma)}\colon \QQ] =12 $, we deduce the extension $L/L^{D^2(\Gamma)}$ has degree at most 11.
So we may apply Lemma \ref{lemma:gp_thry} to deduce that $D^2(\Gamma)$ is a 2-group.
Thus by Lemma \ref{lemma:running_examples_points_generate_ell_group}, any simple object of $\mathcal{C}_{19}$ has its points defined over $L^{D^2(\Gamma)}=F=\QQ(\sqrt{-19}, \alpha)$.
\end{proof}

\subsection{Step 2} \label{subsec:simple_19}
The elliptic curve $J_0(19)$ has minimal Weierstrass equation \[y^2+y=x^3+x^2-9x-15.\]
By Proposition \ref{prop:torsion_is_finite_flat}, \cite[Thm. VI.6.9]{Deligne_Rapoport} and \cite[Cor. 9.7.2]{BLR_book_1990}, the $2$-torsion of $J_0(19)$ gives rise to an object $E \coloneqq J_0(19)[2]$ in $\mathcal{C}_{19}$.

\begin{lemma}
\label{lem:agree_over_Z2}
The finite flat group scheme $E/\ZZ[\frac{1}{19}]$ is simple and $\QQ(E) = \QQ(\sqrt{-19}, \alpha)$.
Its base change $E_{\ZZ_{2}}$ is connected, with connected dual.
In particular, any object $M$ in $\mathcal{C}_{19}$ which satisfies $M_{\QQ_2} \cong E_{\QQ_2}$ also satisfies $M_{\ZZ_{2}} \cong E_{\ZZ_{2}}$.
\end{lemma}

\begin{proof}
An easy computation shows $\QQ(E) = \QQ(\sqrt{-19}, \alpha)$.
The Galois group of the extension $\QQ(E)/\QQ$ is isomorphic to $S_3$.
Hence $E(\overline{\QQ})$ is a faithful 2-dimensional representation of $S_3$ over $\FF_2$.
There is only one such representation, which is irreducible, so by Lemma \ref{lem:correspondence_submodules_subschmes}, $E$ is simple.

Let $E^0$ be the connected component of $E _{ \ZZ_{2}}$.
Its geometric points $E^0(\overline{\QQ}_2)$ give an $\FF_2[\Gamma_{\QQ_2}]$-submodule of $E(\overline{\QQ}_2)$.
As $\Gamma_{\QQ_2}$ acts on $E(\overline{\QQ}_2)$ via a quotient isomorphic to $S_3$,   we find $E^0(\overline{\QQ}_2)$ is either equal to $E(\overline{\QQ}_2)$ or trivial.
In the latter case, the connected-étale sequence \cite[pg. 138 (I)]{Tate_FLT} implies that $E _{ \ZZ_{2}}$ is étale.
However, $2$ ramifies in  $\QQ_2(\sqrt{-19}, \alpha) = \QQ_2(E)$, so $E _{ \ZZ_{2}}$ cannot be étale.
Hence $E _{ \ZZ_{2}}$ is connected.

Furthermore as $J_0(19)$ is principally polarised, $E$ is self-dual \cite[Cor. 27.214]{Goertz_Wedhorn_II}, so its dual is also connected.
That is, $E _{ \ZZ_{2}}$ is biconnected.
By Proposition \ref{prop:unique_biconnected}, all prolongations of $E _{ \QQ_2}$ to $\ZZ_2$ are isomorphic.
Hence the result follows.
\end{proof}

\begin{lemma}
\label{lem:agree_over_38}
Suppose $M$ in $\mathcal{C}_{19}$ satisfies $M(\overline{\QQ}) \cong E(\overline{ \QQ})$ as Galois modules.
Then we also have $M _{ \ZZ[\frac{1}{2p}] }\cong E _{ \ZZ[\frac{1}{2p}]}$.
\end{lemma}

\begin{proof}
As $M _{ \ZZ[\frac{1}{2p}]}$ is étale for any $M$ in $\mathcal{C}_{19}$, it is determined by  the action of the étale fundamental group of $\ZZ[\frac{1}{2p}]$ on its geometric points \cite[pg. 137]{Tate_FLT}.
The étale fundamental group of $\ZZ[\frac{1}{2p}]$ is a quotient of $\Gamma_\QQ$, so the result follows from the isomorphism  $M(\overline{\QQ}) \cong E(\overline{ \QQ})$ of $\Gamma_\QQ$-modules.
\end{proof}

\begin{theorem}
\label{thm:simple_group_schemes_in_C_19}
The simple group schemes in $\mathcal{C}_{19}$ are $\underline{\ZZ/2\ZZ}$, $\mu_2$ and $E$.
\end{theorem}

\begin{proof}
Let $M$ be a simple group scheme in $\mathcal{C}_{19}$.
Any group scheme over $\QQ$ is étale \cite[pg. 138, (II)]{Tate_FLT} and thus determined by its associated Galois module \cite[Pg. 137]{Tate_FLT}.
We thus first look to determine $M _{ \QQ}$.
By Proposition \ref{prop:19_geom_points}, the geometric points of $M$ are defined over $\QQ(\sqrt{-19}, \alpha)$.
The Galois group  of the extension $\QQ(\sqrt{-19}, \alpha)/\QQ$ is isomorphic to $S_3$.
There are only two irreducible modules for $S_3$ in characteristic two, the trivial one, and the standard representation of dimension two.
Hence either $M$ has order two or four, and moreover either $M _{ \QQ }\cong (\underline{\ZZ/2\ZZ}) _{ \QQ}$ or $E _{ \QQ}$.

In the case $M$ has order two, Proposition \ref{prop:grp_schms_order_2} shows $M$ is isomorphic to either $\underline{\ZZ/2\ZZ}$ or $\mu_2$.

Suppose  $M _{ \QQ} \cong E _{ \QQ}$.
We will use the equivalence of categories in Proposition \ref{prop:equivalence_categories_prolongations} to show $M \cong E$.
Under this equivalence $M$ corresponds to $(M _{ \ZZ_2},  M _{ \ZZ[\frac{1}{38}]}, \mathrm{id}_{M _{ \QQ_2}})$ and $E$ to $(E _{ \ZZ_2},  E _{ \ZZ[\frac{1}{38}]}, \mathrm{id}_{E _{ \QQ_2}})$.
Lemma \ref{lem:agree_over_Z2} and \ref{lem:agree_over_38} show that there exist isomorphisms $\varphi \colon M _{ \ZZ_2 }\rightarrow E _{ \ZZ_2 }$ and $\varphi' \colon M _{ \ZZ[\frac{1}{38}] }\rightarrow E _{ \ZZ[\frac{1}{38}]}$.

Furthermore, as $M(\overline{\QQ}_2) \cong E(\overline{\QQ}_2)$ are absolutely irreducible $\FF_2[\Gamma_{\QQ_2}]$-modules, we have $\Hom_{\FF_2[\Gamma_{\QQ_2}]}(M(\overline{\QQ}_2), E(\overline{\QQ}_2)) \cong \FF_2$.
Using that $M$ and $E$ are étale over $\QQ_2$, we thus find there is a unique isomorphism $M _{ \QQ_2} \rightarrow E _{ \QQ_2}$ of $\QQ_2$-group schemes.
As base change sends isomorphisms to isomorphisms, the base changes of $\varphi$ and $\varphi'$ to $\QQ_2$ are equal.
It follows that the diagram below
\begin{center}
    \begin{tikzcd}
E _{ \QQ_2 }\arrow[r, "\mathrm{id}_{E _{ \QQ_2}}"] \arrow[d, "\varphi _{ \QQ_2}"]
&  E _{ \QQ_2 }\arrow[d, "\varphi' _{ \QQ_2}"] \\
M _{ \QQ_2 }\arrow[r, "\mathrm{id}_{M _{ \QQ_2}}"]
& M _{ \QQ_2}
\end{tikzcd}
\end{center}
is commutative.
Whence $M \cong E$.
\end{proof}

\subsection{Step 3}
Recall that for finite flat group schemes $G,H$ over $\ZZ[\frac{1}{19}]$, we interpret $ \Ext^1_{\ZZ[\frac{1}{19}]}(G,H)$ as the set of extensions of $G$ by $H$ up to equivalence in the category of finite flat group schemes over $\ZZ[\frac{1}{19}]$ and this set is equipped with a natural abelian group law as explained in Section \ref{subsection:extensions}. 
Likewise, if $G $ and $H$ belong to $\mathcal{C}_{19}$, then we interpret $ \Ext^1_{\mathcal{C}_{19}}(G,H)$ as the subgroup of extensions of $G$ by $H$ in  $\mathcal{C}_{19}$.

\begin{proposition}
\label{prop:19_extensions_with_E_and_order_2_stuff}
We have
\[ \Ext^1_{\mathcal{C}_{19}}(E,\underline{\ZZ/2\ZZ}) =0.\]
\end{proposition}

\begin{proof}
Let $M$ be an extension of $E$ by $\underline{\ZZ/2\ZZ}$ in $\mathcal{C}_{19}$.
That is, we have an exact sequence
\[0 \rightarrow \underline{\ZZ/2\ZZ} \rightarrow M \rightarrow E \rightarrow 0.
\]
Let $L = \QQ(M)$ and $F = \QQ(E)$.
We shall show $L=F$ by proving that $L/F$ is unramified and abelian.
As $M$ is étale over $\ZZ[\frac{1}{38}]$, the extension $L/F$ is unramified outside of $38$.

By Lemma \ref{lem:agree_over_Z2}, $E_{\ZZ_2}$ is connected.
Hence the connected-étale sequence \cite[pg. 138 (I)]{Tate_FLT} furnishes a section for the above exact sequence base changed to $\ZZ_2$.
Thus $M_{ \ZZ_2} \cong (\underline{\ZZ/2\ZZ} \times E)_{\ZZ_2}$, in particular $\Ext^1_{\ZZ_2}(E,\underline{\ZZ/2\ZZ})=0$.
From which we deduce that $\Gal(\overline{\QQ}_2/\QQ_2(E))$ acts trivially on $M(\overline{\QQ}_2)$, and $M$ is annihilated by $2$, from applying Corollary \ref{Cor:if_n_kills_the_base_change_then_it_kills_you_too}.
Thus $L/F$ is unramified at primes above $2$ and the exponent of $\Gal(L/F)$ divides 2 by Lemma \ref{lemma:trivial_by_trivial_implies_order_divides_n}.
Hence $L/F$ is an abelian extension.

As $F$ has class number one, it remains to show $L/F$ is unramified at primes above 19.
Recall from Lemma \ref{lem:agree_over_Z2} that $\QQ_{19}(E) = \QQ_{19}(\sqrt{-19}, \alpha)= \QQ_{19}(\sqrt{-19})$.
So $\QQ_{19}(E)/ \QQ_{19}$ is a totally ramified extension of degree two.
By Corollary \ref{Cor:properties_of_annihilated_by_ell_with_Fontaine}, the inertia group of the extension $\QQ_{19}(M)/\QQ_{19}$ has order 2.
Thus $\QQ_{19}(M)/\QQ_{19}(E)$ is unramified.
Hence $L/F$ is everywhere unramified and so $L=F$.

Using that $L=F$, that $M$ is annihilated by 2 and the equivalence of categories for finite étale group schemes \cite[pg. 137]{Tate_FLT}, we see the image of $M$ in $\Ext^1_{\ZZ[\frac{1}{38}]}(E,\underline{\ZZ/2\ZZ})$ is zero.
As $E(\overline{\QQ}_2)$ is a simple 2-dimensional module over $\FF_2$ for $\Gal(\QQ_2(E)/\QQ_2) \cong S_3$, using étaleness we find $\Hom_{\QQ_2}(E,\underline{\ZZ/2\ZZ}) =0$.
It thus follows from the Mayer-Vietoris sequence in Theorem \ref{thm:Mayer_Vietoris} that $M$ is the trivial extension.
\end{proof}

\begin{lemma}
\label{lemma:endos_are_Z/2Z}
        Let $R$ be one of $\QQ,\QQ_2, \ZZ_2,\ZZ[\frac{1}{19}], \ZZ[\frac{1}{38}]$.
        Then $\End_R(E) \cong \FF_2$.
\end{lemma}

\begin{proof}
As both the zero and identity map belong to $\End_R(E)$ we find $\#\End_R(E) \geq 2$.
Moreover, by étaleness, for $R \in \{\QQ,\QQ_2, \ZZ[\frac{1}{38}]\}$ we see $\End_R(E) \cong \FF_2$ by considering the action of Galois.
It then follows from Proposition \ref{prop:unique_extension_of_homs} that $\End_R(E) \cong \FF_2$ also for $R\in \{\ZZ_2,\ZZ[\frac{1}{19}]\}$.
\end{proof}

\begin{proposition}
\label{prop:19_extensions_of_E_by_itself}
    $\Ext^1_{\mathcal{C}_{19}}(E,E) \cong \FF_2$ and is generated by $J_0(19)[4]$.
\end{proposition}

\begin{proof}
By \cite[Prop. A.2]{Brumer_Kramer_Certain_ab_vars_bad_at_one_prime} we have an exact sequence
\[0 \rightarrow \Ext^1_{\ZZ[\frac{1}{19}],[2]}(E,E)  \rightarrow  \Ext^1_{\ZZ[\frac{1}{19}]}(E,E) \rightarrow \End_{\QQ}(E)
\]
where  $\Ext^1_{\ZZ[\frac{1}{19}],[2]}(E,E)$ denotes the subgroup of extensions of $E$ by itself which are annihilated by 2 as group schemes.
We denote the intersection of this subgroup with $\Ext^1_{\mathcal{C}_{19}}(E,E)$ by $\Ext^1_{\mathcal{C}_{19},[2]}(E,E)$.
The above exact sequence implies $\Ext^1_{\mathcal{C}_{19},[2]}(E,E)$ has index at most 2 in $\Ext^1_{\mathcal{C}_{19}}(E,E)$.
In fact, as $\Ext^1_{\mathcal{C}_{19}}(E,E)$ contains $J_0(19)[4]$ (recall $E = J_0(19)[2]$) we see the index is exactly 2.

Let $M \in \Ext^1_{\mathcal{C}_{19},[2]}(E,E)$ and $L\coloneqq \QQ(M)$.
To complete the proof, it suffices to show $M \cong E \times E$, i.e., $M$ is split.
Recall that $F\coloneqq \QQ(E) = \QQ(\sqrt{-19}, \alpha)$ has class number one and $(\alpha)$ is the unique prime ideal above 2.
By Corollary \ref{Cor:properties_of_annihilated_by_ell_with_Fontaine}, the inertia groups above 19 in $\Gal(L/\QQ)$ have order dividing 2, so  $L/F$ is unramified outside of 2.
Using again that $M$ is annihilated by 2, we find the exponent of $\Gal(L/F)$ divides 2 by Lemma \ref{lemma:trivial_by_trivial_implies_order_divides_n}.
Thus  $\Gal(L/F)$ is  abelian.

Recall that $M_{\ZZ_2}$ is an extension of $E_{\ZZ_2}$ by itself annihilated by 2 and $E_{\ZZ_2}$ is connected with connected dual by Lemma \ref{lem:agree_over_Z2}.
Hence we may analyse $M_{\ZZ_2}$ using results from \cite[Section 6]{Schoof_cyclotomic}.

Let $K \coloneqq \QQ_2(E)$ and $P \in E(\overline{\QQ}_2)$.
The field $K_P$ generated over $\QQ_2$ by the geometric points of $M_{\ZZ_2}$ which map to  $P$ is a Galois extension of $K$ with degree dividing $4$  \cite[Prop. 6.3]{Schoof_cyclotomic}\footnote{See also the errata \url{https://www.mat.uniroma2.it/~schoof/erratacyc}}.
The extension $K_P/K$ is either trivial, unramified or totally ramified \cite[Prop. 6.4]{Schoof_cyclotomic}.
Moreover, in the totally ramified case, $\Gal(K_P/K) \cong C_2^2$ and its non-trivial characters have Artin conductor $(\sqrt[3]{-2})^2 = (\alpha)^2$.

Let us suppose for a contradiction that there exists some $P \in E(\overline{\QQ}_2)$ such that $K_P/K$ is ramified.
Then $L/F$ is non-trivial and thus $\Gal(L/F)$ has exponent 2.
In particular, $K_P$ is the completion of some subextension $F_P/F$ of $L/F$.
It follows that the Artin conductor of any non-trivial character of $\Gal(F_P/F)$ is $(\alpha)^2$ \cite[Thm. VII.11.6]{Neukirch_ANT_book}. 
This implies \cite[Thm. VII.11.10]{Neukirch_ANT_book} that the field cut out by any non-trivial character of $\Gal(F_P/F)$ has conductor $(\alpha)^2$. 

The fundamental units of $F$ generate $(\OO/2\OO)^*$ and $F$ has class number one.
It follows that the ray class group of $F$ of modulus $(2) = (\alpha)^3$ is trivial.
In particular, any abelian extension of $F$ unramified outside of $(\alpha)$ must be totally ramified at $(\alpha)$ of conductor at least $(\alpha)^4$.
This provides us with the sought after contradiction.
Thus for every $P \in E(\overline{\QQ}_2)$ the extension $K_P/K$ is either trivial or unramified.

We deduce that $L/F$ is unramified.
However, $L/F$ is an abelian extension and $F$ has class number one, so $L=F$.
It follows from étaleness that $M_{\ZZ[\frac{1}{38}]}$ and $M_{\QQ_2}$ are both split.

By Lemma \ref{lem:agree_over_Z2}, $E_{\ZZ_2}$ is the unique prolongation of $E_{\QQ_2}$ over $\ZZ_2$, thus by Proposition \ref{prop:base_change_on_extension_groups} we have $M_{\ZZ_2} \cong E_{\ZZ_2} \times E_{\ZZ_2}$.
As a consequence of Lemma \ref{lemma:endos_are_Z/2Z}, the Mayer-Vietoris sequence in Theorem \ref{thm:Mayer_Vietoris} simplifies to 
\[0 \rightarrow \Ext^1_{\ZZ[\frac{1}{19}]}(E,E)  \rightarrow \Ext^1_{\ZZ_2}(E,E) \times \Ext^1_{\ZZ[\frac{1}{38}]}(E,E) \rightarrow \Ext^1_{\QQ_2}(E,E).
\]
We thus deduce $M$ is split.
\end{proof}
\subsection{Step 4} \label{sec:Classification_2-divisible_groups}
We are now in a position to complete the classification of semistable abelian varieties over $\mathbb{Q}$ with good reduction outside $19$. For this, we show that the hypotheses of Theorem \ref{thm:Schoof_classification_p_divisible groups} are satisfied in our setting. 

\begin{theorem}
\label{thm:19_classification}
    Let $A/\QQ$ be a semistable abelian variety with good reduction outside of 19.
    Then $A$ is isogenous to $J_0(19)^g$ for some $g \geq 1$.
\end{theorem}

\begin{proof}
Lemma \ref{lem:running_example_extensions_mu_2_by_Z_mod_2_Z_over_1/3} and
     Proposition  \ref{prop:19_extensions_with_E_and_order_2_stuff} allow us to apply Proposition \ref{prop:filtration_has_no_Z/lZ_nor_mu_l} in combination with Theorem \ref{thm:simple_group_schemes_in_C_19} to deduce there exists a filtration of $A[\ell^n]$:
    \[A[\ell^n] = G_0 \supseteq G_1 \supseteq \ldots \supseteq G_m = 0\]
    such that the successive quotients $G_i/G_{i+1}$ are isomorphic to $E = J_0(19)[2]$.
  
    It is easy to check the action of $\GQ$ on $T_2J_0(19)$ induces a surjective homomorphism $\GQ \rightarrow \GL_2(\ZZ_2)$.
    Thus $\End_{\GQ}(T_2J_0(19)) \cong \ZZ_2$.
    The endomorphism ring of $\mathcal{X}_2$, the $2$-divisible group associated to the Néron model of $J_0(19)$ over $\ZZ[\frac{1}{19}]$, injects into $\End_{\GQ}(T_2J_0(19))$.
    Thus $\End(\mathcal{X}_2) = \ZZ_2$. 

    By applying the functor $\Hom_{\ZZ[\frac{1}{19}]}(E,-)$ to the exact sequence $0\to E \to J_0(19)[4] \to E \to 0$, we obtain the following long exact sequence:
    \[
    0 \rightarrow \Hom_{\ZZ[\frac{1}{19}]}(E,E) \rightarrow \Hom_{\ZZ[\frac{1}{19}]}(E,J_0(19)[4])
    \rightarrow \Hom_{\ZZ[\frac{1}{19}]}(E,E)
    \xrightarrow{\delta} \Ext^1_{\mathcal{C}_{19}}(E,E).
    \]
    Note that the image of any $\varphi \in \Hom_{\ZZ[\frac{1}{19}]}(E,J_0(19)[4])$ is annihilated by 2, since $E$ is.
    In particular, its image lands in $J_0(19)[2] = E$.
    Thus the above map $\Hom_{\ZZ[\frac{1}{19}]}(E,E) \rightarrow \Hom_{\ZZ[\frac{1}{19}]}(E,J_0(19)[4])$ is an isomorphism.
    
    By Lemma \ref{lemma:endos_are_Z/2Z} and Proposition \ref{prop:19_extensions_of_E_by_itself}, we have $\Hom_{\ZZ[\frac{1}{19}]}(E,E) \cong \Ext^1_{\mathcal{C}_{19}}(E,E) \cong \FF_2$ as groups.
    It thus follows from the above exact sequence that $\delta$ is an isomorphism.
    Applying Theorem \ref{thm:Schoof_classification_p_divisible groups} in combination with Faltings' Isogeny Theorem \cite[Kor. 2 (ii)]{Faltings_1983} we deduce $A$ is isogenous to $J_0(19)^g$ for some $g \geq 1$ as required.
\end{proof}

\section{Simple objects of order 4 in $\mathcal{C}_{p}$}
\label{sec:simple_objects_of_order_4}
The simple objects in the category $\mathcal{C}_{11}$ are isomorphic to one of $\underline{\ZZ/2\ZZ}, \mu_2$ and $\mathcal{E}[2]$, the 2-torsion of the Néron model $\mathcal{E}$ of the CM elliptic curve  $y^2+y=x^3-x^2-7x+10$ over $\ZZ[\frac{1}{11}]$, see \cite[Prop. 7.1]{Schoof_one_bad_prime} and the discussion following \cite[Prop. 7.2]{Schoof_one_bad_prime}.
However, as remarked in \cite[pg. 861]{Schoof_one_bad_prime}, the group scheme $\mathcal{E}[4]$ does not belong to $\mathcal{C}_{11}$.

Thanks to the results in Section \ref{sec:19}, one may verify a similar situation occurs for the prime $19$.
Namely, the 2-torsion of the Néron model $\mathcal{E}'$ of the CM elliptic curve  $y^2+y=x^3-38x+90$ over $\ZZ[\frac{1}{19}]$ is isomorphic to $J_0(19)[2]$, but $\mathcal{E}'[4]$ does not belong to $\mathcal{C}_{19}$.
Proposition \ref{prop:PGR_2-torsion_can_give_objects_in_C_p} and Corollaries \ref{cor:objects_of_C_p_from_CM_ECs} and \ref{cor:CM_ECs_examples} show this is an example of a more general phenomenon.

\begin{proposition} 
\label{prop:PGR_2-torsion_can_give_objects_in_C_p}
    Suppose $p > 3$ is prime.
    Let $E$ be an elliptic curve over $\mathbb{Q}$ with good reduction outside of $p$ and potentially good reduction at $p$.
    Suppose the valuation of $\Delta_{\mathrm{min}}(E)$ at $p$ is odd.

    Let $\mathcal{E}\to \Spec(\ZZ[\frac{1}{p}])$ be the Néron model of $E$ over $\ZZ[\frac{1}{p}]$.
    Then $\mathcal{E}[2]$ is an object of $\mathcal{C}_{p}$, but $\mathcal{E}[4]$ is not.
\end{proposition}

\begin{proof}
    Throughout the proof $\ell$ shall denote a prime different from $p$.
    As $E$ has potentially good reduction at $p$ and the valuation of $\Delta_{\mathrm{min}}(E)$ at $p$ is odd, any fixed inertia group $I_p$ at $p$, acts on $T_\ell E$ for through a cyclic quotient of order 4 \cite[pg. 312]{Serre_proprieties_galoisiennes}.
    Furthermore, the kernel of the action of $I_p$ on $T_\ell E$ is the same for all $\ell$ and the corresponding character has values in $\ZZ$ and is independent of $\ell$ \cite[Thm. 2 (ii)]{Serre_Tate}.

    For $\ell \geq 3$, the reduction map $\GL_2(\ZZ_\ell) \rightarrow \GL_2(\FF_\ell)$ is injective on finite subgroups.
    In particular, the image of $I_p$ in $\Aut(E[\ell])$ is cyclic of order 4 for any $\ell \geq 3$.
    Suppose $\ell$ satisfies $\ell \equiv 3 \mod{4}$.
    Let $\tau$ denote a generator of the image of $I_p$ in $\Aut(E[\ell])$.
   In particular, $\tau$ has at least one eigenvalue $a$ of order 4 and, as $x^2+1$ is irreducible over $\FF_\ell$, the conjugate of $a$ is also an eigenvalue of $\tau$.
   Thus $\tau$ has trace zero.
   Since this holds for infinitely many primes, we deduce by the above that any element of order 4 in the image of $I_p$ in $\Aut(T_\ell E)$ has trace zero for any $\ell$.

   The reduction map $\GL_2(\ZZ_2) \rightarrow \GL(\ZZ/4\ZZ)$ is injective on finite subgroups, so we find the image of $I_p$ in $\Aut(E[4])$ is generated by an element of order $4$ with trace equal to zero.
   Since any element $\sigma$ of order 4 in $\GL_2(\ZZ/4\ZZ)$ satisfying $(\sigma-1)^2=0$ has trace equal to 2, we find $\mathcal{E}[4]$ does not belong to $\mathcal{C}_p$.

   The group $\GL_2(\FF_2)$ is isomorphic to $S_3$.
   In particular, the image of $I_p$ in $\Aut(E[2])$ has order dividing 2.
   Any element $\sigma$ of order dividing 2 in $\GL_2(\FF_2)$ satisfies $(\sigma-1)^2=0$.
   Thus, by Proposition \ref{prop:torsion_is_finite_flat}, $\mathcal{E}[2]$ belongs to $\mathcal{C}_p$. 
\end{proof}

\begin{corollary}
\label{cor:objects_of_C_p_from_CM_ECs}
    Let $p> 3$ be such that $\QQ(\sqrt{-p})$ has class number one.
    Then there exists an elliptic curve $E/\QQ$ with CM by $\QQ(\sqrt{-p})$ for which $\Delta_{\mathrm{min}}(E) = -p^3$.
    
    Let  $\mathcal{E}$ denote the Néron model of $E$ over $\ZZ[\frac{1}{p}]$.
    Then $\mathcal{E}[2]$ belongs to $\mathcal{C}_{p}$, but $\mathcal{E}[4]$ does not.
\end{corollary}

\begin{proof}
    The existence of such an elliptic curve is given by \cite[Thm. 12.2.1]{Gross_CM_book}.
    As any CM elliptic curve has potentially good reduction \cite[Thm II.6.1]{Silverman_II}, the result follows from Proposition \ref{prop:PGR_2-torsion_can_give_objects_in_C_p}.
\end{proof}

\begin{corollary}
\label{cor:CM_ECs_examples}
    Let $p\in \{11,19,43,67,163\}$.
    Then $\mathcal{C}_p$ contains at least three pairwise non-isomorphic simple objects.
    Namely $\underline{\ZZ/2\ZZ}, \mu_2$ and $\mathcal{E}[2]$ where $\mathcal{E}$ is as in Corollary \ref{cor:objects_of_C_p_from_CM_ECs}.
\end{corollary}

\begin{proof}
    Let $E$ be an elliptic curve as in the statement of Corollary \ref{cor:objects_of_C_p_from_CM_ECs}.
    The ray class field of $K \coloneqq\QQ(\sqrt{-p})$ of modulus $(2)$ is given by $K(E[2])$, see \cite[Thm. II.5.6]{Silverman_II}.
    As $-p \equiv 5 \mod{8}$, 2 is inert in $K$.
    Moreover, the unit group of $\mathcal{O}_K$, the ring of integers of $K$, equals $\{\pm 1 \}$.
    Recall that there is an exact sequence 
    \[\OO_K^* \rightarrow (\OO_K/2\OO_K)^* \rightarrow \mathrm{Cl}_{(2)}(K) \rightarrow\mathrm{Cl}(K) \rightarrow0\]
    where $\mathrm{Cl}_{(2)}(K)$ denotes the ray class group of modulus $(2)$ and $\mathrm{Cl}(K)$ the class group of $K$ \cite[Prop. 3.2.3]{Cohen_advanced_computational_ANT}.
    By the above, this becomes 
     \[\{\pm 1\} \rightarrow C_3 \rightarrow \mathrm{Cl}_{(2)}(K)  \rightarrow0.\]
     We deduce that $K(E[2])/K$ is a degree 3 extension totally ramified at $(2)$.
     It follows that $E[2]$ is an irreducible $\Gamma_\QQ$-module.
     Hence by Lemma \ref{lem:correspondence_submodules_subschmes}, $\mathcal{E}[2]$ is simple.

     Corollary \ref{cor:objects_of_C_p_from_CM_ECs} thus allows us to conclude $\mathcal{E}[2]$ is a simple object of $\mathcal{C}_p$.
     As $\underline{\ZZ/2\ZZ}, \mu_2$ both have order 2, and $\Gamma_\QQ$ acts trivially on their geometric points, we easily deduce that they are simple objects of $\mathcal{C}_p$.
\end{proof}

Using the techniques developed in Section \ref{sec:strategy_prolongations} we may expand upon the previous examples.
First we shall need a lemma:

\begin{lemma}
\label{lem:family_of_simple_grp_schemes_order_4_field_of_points}
    Let $p \equiv 3 \mod{8}$ be a prime greater than 3.
    Suppose the class number of $\QQ(\sqrt{-p})$ is not divisible by 3.
    Then there exists a degree three extension $F$ of $\QQ(\sqrt{-p})$ ramified exactly at 2 and Galois over $\QQ$ with $\Gal(F/\QQ) \cong S_3$.
\end{lemma}

\begin{proof}
    Let $K = \QQ(\sqrt{-p})$.
    Let $\OO= \ZZ[\frac{1+\sqrt{-p}}{2}]$ the maximal order in $\QQ(\sqrt{-p})$.
    As $p>3$, we have $\OO^* = \{\pm 1\}$.
    The prime 2 is inert in $\QQ(\sqrt{-p})$ since $-p \equiv 5 \mod{8}$.
    Recall that $\mathrm{Cl}_{2}(K)$ the ray class group of modulus 2 and the class group $\mathrm{Cl}(K)$ fit in the exact sequence \cite[Prop. 3.2.3]{Cohen_advanced_computational_ANT}
    \[\OO^* \rightarrow (\OO/2\OO)^* \rightarrow \mathrm{Cl}_{2}(K) \rightarrow \mathrm{Cl}(K) \rightarrow 1\]
    which when combined with the above shows the below sequence is exact
    \[1 \rightarrow C_3 \rightarrow \mathrm{Cl}_{2}(K) \rightarrow \mathrm{Cl}(K) \rightarrow 1.\]
    Let $F_{(2)}$ denote the ray class field of modulus 2.
    Note that $F_{(2)}$ is Galois over $\QQ$.
    By assumption the order of $\mathrm{Cl}(K)$ is coprime to 3.
    Thus by the Chinese Remainder Theorem $\mathrm{Cl}_{2}(K) \cong C_3 \times \mathrm{Cl}(K)$.
    Class Field Theory then implies $\Gal(F_{(2)}/K) \cong C_3 \times \mathrm{Cl}(K)$, which being of index two is a normal subgroup of $\Gal(F_{(2)}/\QQ)$.
    As conjugation preserves the order of a given element, we find that there is a normal subgroup of $\Gal(F_{(2)}/\QQ)$ isomorphic to $\mathrm{Cl}(K)$ contained in $\Gal(F_{(2)}/K)$.
    Let $F$ denote the fixed field of this normal subgroup.
    Then $F$ is a degree 3 extension of $K$ ramified exactly at 2.

    We are left to show that $\Gal(F/\QQ) \cong S_3$.
    Observe that if $\Gal(F/\QQ) \cong C_6$, then there exists a degree 3 abelian extension of $\QQ$ ramified only at 2.
    But by the Kronecker-Weber Theorem no such extension exists.
    Hence $\Gal(F/\QQ) \cong S_3$.
\end{proof}

\begin{proposition}
\label{prop:family_of_simple_grp_schemes_order_4_existence}
    Let $p \equiv 3 \mod{8}$ be a prime greater than 3.
    Suppose the class number of $\QQ(\sqrt{-p})$ is not divisible by 3.
    Then there exists a unique simple finite flat group scheme $G$ of order 4 over $\ZZ[\frac{1}{p}]$ satisfying $\QQ(G) = F$ where $F$ is the field from Lemma \ref{lem:family_of_simple_grp_schemes_order_4_field_of_points}.
    Moreover, $G$ belongs to $\mathcal{C}_p$.
\end{proposition}

\begin{proof}
    By Lemma \ref{lem:family_of_simple_grp_schemes_order_4_field_of_points} the extension $F/\QQ(\sqrt{-p})$ is ramified exactly at 2 and Galois over $\QQ$ with $\Gal(F/\QQ) \cong S_3$.
    As $S_3$ has a unique 2 dimensional irreducible representation over $\FF_2$, there exists a unique étale group scheme $G'$ of order 4 over $\ZZ[\frac{1}{2p}]$ annihilated by 2 satisfying $\QQ(G')=F$.
    Moreover, as any finite flat group scheme of order 4 over $\ZZ[\frac{1}{2p}]$ is étale, this is the unique group scheme of order 4 over $\ZZ[\frac{1}{2p}]$ which is annihilated by 2 with field of points being $F$.

    Let $F_2$ denote the completion of $F$ at the unique prime above 2.
    As the residue degree of $2$ in $F$ is 2 and its ramification degree equals 3, we have that $[F_2\colon \QQ_2 ]= 6$.
    Since the degree 2 subextension of $F_2$ is unramified, we find it is isomorphic to $\QQ_2(\zeta_3)$ as this is the unique unramified extension of that degree.
    Moreover, as $F_2/\QQ_2(\zeta_3)$ is totally tamely ramified, we deduce $F_2= \QQ_2(\zeta_3, \sqrt[3]{2u})$ \cite[Prop. II.7.7]{Neukirch_ANT_book}
    for some $u \in \ZZ_2[\zeta_3]^*$.
    By Hensel's Lemma, $u$ is a cube in $\QQ_2(\zeta_3)$, thus $F_2= \QQ_2(\zeta_3, \sqrt[3]{2})$.
    
    The action of $\Gamma_{\QQ_2}$ on $G'_{\QQ_2}(\overline{\QQ}_2)$ and $\mathcal{E}[2](\overline{\QQ}_2)$ where $\mathcal{E}[2]$ is as in Corollary \ref{cor:CM_ECs_examples} (for any particular value of $p$ considered there) factors through $\Gal(F_2/\QQ_2)$.
    Since $G'_{\QQ_2}$ and $\mathcal{E}[2]_{\QQ_2}$  are étale and $S_3$ has a unique 2 dimensional irreducible representation over $\FF_2$, we find $G'_{\QQ_2} \cong \mathcal{E}[2]_{\QQ_2}$.
    Thus $G'_{\QQ_2}$ has a prolongation $G''$ to $\ZZ_2$ given by $\mathcal{E}[2]_{\ZZ_2}$.
    Note moreover that as $G''$ is biconnected it is the unique prolongation of $G'_{\QQ_2}$ to $\ZZ_2$ by Proposition \ref{prop:unique_biconnected}.

    As the action of $S_3$ on its irreducible 2-dimensional module over $\FF_2$ is absolutely irreducible, we deduce that there is a unique $\FF_2[\Gal(F_2/\QQ_2)]$-automorphism of $G'(\overline{\QQ}_2)$.
    The equivalence of categories given in Proposition \ref{prop:equivalence_categories_prolongations} thus shows that there is a unique prolongation $G$ of $G'$ to $\ZZ[\frac{1}{p}]$.
\end{proof}

\printbibliography

\end{document}